\documentclass[11pt]{article}
\usepackage{latexsym,amsmath,amssymb,amsthm}
\usepackage[usenames]{color}
\usepackage{graphicx}

\newfont{\bb}{msbm10 at 12pt}
\newfont{\bbp}{msbm10 at 10pt}

\def\r{\hbox{\bb R}}

\def\n{\hbox{\bb N}}
\def\np{\hbox{\bbp N}}

\def\zz{\hbox{\bb Z}}

\def\s{\hbox{\bb S}}

\def\amb{\mathcal{M}}
\def\anb{\mathcal{N}}

\newcommand{\cR}{{\color{red}\blacksquare}}
\newcommand{\cB}{{\color{blue}\blacklozenge}}
\newcommand{\campo}{\mathfrak{X}}
\newcommand{\camb}{\overline{\nabla}}

\newcommand{\abs}[1]{\left\vert #1 \right\vert}
\newcommand{\set}[1]{\left\{#1\right\}}
\newcommand{\meta}[2]{\langle #1,#2 \rangle }

\newcommand{\Kb}{\overline{K}}
\newcommand{\Rb}{\overline{R}}
\newcommand{\Ricb}{\overline{Ric}}
\newcommand{\Sb}{\overline{Sc}}
\newcommand{\derb}[3]{\left.\frac{d #1}{d #2}\right\vert _{#3} }

\numberwithin{equation} {section}

\newtheorem{lemma}{Lemma}[section]
\newtheorem{remark}{Remark}[section]
\newtheorem{definition}{Definition}[section]
\newtheorem{theorem}{Theorem}[section]
\newtheorem{proposition}{Proposition}[section]

\theoremstyle{}

\begin{document}

\begin{center}
\rule{15cm}{1.5pt} \vspace{.3cm}

{\Large \bf  Frankel Property, Maximum Principle at Infinity\\[3mm]  and Splitting Theorems\\[3mm]  for complete minimal hypersurfaces} 

\vspace{0.5cm}
%\footnote{The author is partially supported by Spanish MICINN-FEDER (Grant PID2020-118137GB-I00 and Grant RyC-2016-19359); Junta de Andaluc\'{i}a (Grant P18-FR-4049 and Grant FQM325).}
{\large Jos$\acute{\text{e}}$ M. Espinar$\,^\dag$, Harold Rosenberg$\,^\ddag$}\\ 
\vspace{0.3cm} \rule{15cm}{1.5pt}
\end{center}

\begin{center}
{\footnotesize $^{\dag}$Departamento de Geometr\'ia y Topolog\'ia, Facultad de Ciencias, Universidad de Granada, Avenida Fuentenueva s/n, 18001, Granada, Spain\\Email: jespinar@ugr.es\\
$^{\ddag}$Email: rosenbergharold@gmail.com}
\end{center}

\begin{abstract} 
In this paper, we study complete minimal hypersurfaces in Riemannian \(n\)-manifolds \(\amb^n\) for dimensions \(4 \leq n \leq 7\), and we obtain some results in the spirit of known work for $n=3$.

Key contributions include extending the work of Anderson and Rodríguez to higher dimensions. Specifically, we show that in four dimensional manifolds with nonnegative sectional curvature and positive scalar curvature, two disjoint properly embedded minimal hypersurfaces bound a slab isometric to the product of one hypersurface with an interval. %These results offer potential applications and extensions to the study of non-compact $4-$manifolds with nonnegative sectional curvature and uniformly positive scalar curvature.

Our results are grounded in a maximum principle at infinity for two-sided, parabolic, properly embedded minimal hypersurfaces in complete Riemannian manifolds of bounded geometry, generalizing the work of Mazet in dimension three to higher dimensions. We also leverage the recent classification of complete two-sided stable minimal hypersurfaces by Chodosh, Li, and Stryker.
\end{abstract}



\section{Introduction}

In 1897, J. Hadamard  \cite{JHad97} proved that two complete and embedded geodesics on a strictly convex $2-$sphere must intersect if one of them is compact. The closure of an embedded complete geodesic is a geodesic lamination. If a geodesic lamination of a strictly convex sphere were not one compact geodesic, then it would contain a stable geodesic.  Using the formula for the second variation of arc length and a log-cut-off trick, one sees there is no such stable geodesic.  Hence an embedded complete geodesic on such a $2-$sphere (e.g., an ellipsoid) is a compact (circle) geodesic. In the Appendix, we provide another proof of this using the Gauss-Bonnet Theorem. When the ambient dimension is three, we have:

\begin{quote}
{\bf Theorem \cite{HRos01}.} {\it Let $\amb ^3$ be a closed orientable manifold with positive Ricci curvature. If $\Sigma$ is an injectively immersed complete orientable minimal surface of bounded curvature in $\amb ^3$, then $\Sigma$ is compact.}
\end{quote}

The idea behind these theorems is that when the closure of a minimal hypersurface $\Sigma$ is a minimal lamination $\mathcal L$, then $\Sigma$ is properly embedded or there is a limit leaf $L \in \mathcal L$ ($L$ maybe $\Sigma$) that is stable. To exclude $L$, we might ask ourself three different questions:
\begin{itemize}
\item When do two minimal hypersurfaces necessarily intersect?
\item When do two minimal hypersurfaces satisfy a Maximum Principle at Infinity: if $\Sigma _1\cap \Sigma _2 = \emptyset$ then the distance between $\Sigma _1 $ and $\Sigma _2$ is bounded away from zero? 
\item When does the ambient geometry determine the non-existence of stable two-sided minimal hypersurfaces?
\end{itemize}

T. Frankel \cite{TFra66} considered a complete $n-$manifold $\amb ^n$ with positive Ricci curvature and two minimal hypersurfaces $\Sigma _1$ and $\Sigma _2$ in $\amb ^n$, $\Sigma _1$ closed (compact with no boundary), $\Sigma _2$ immersed and $ \Sigma _2$ proper ($\overline{\Sigma}_2 = \Sigma _2$). He proved $\Sigma _1 \cap \Sigma _2 \neq \emptyset $. C. Croke and B. Kleiner \cite{CCroBKle92} considered two closed orientable embedded minimal hypersurfaces in an orientable complete $\amb ^n$ with non-negative Ricci curvature. They proved that if $\Sigma _1$ does not intersect $\Sigma _2$ then they are both totally geodesic and either (i) or (ii) holds: 
\begin{enumerate}
\item[(i)] $\Sigma _1 \cup \Sigma _2$ bounds a domain in $\amb $ isometric to $\Sigma _1 \times [0, {\bf d}]$, with the product metric, ${\bf d}:= {\rm dist}_{\amb}(\Sigma _1, \Sigma_2)$;
\item[(ii)] $\amb ^n$ is a mapping torus; i.e., $\amb $ is isometric to $\Sigma _1 \times [0, {\bf d}]$, with the product metric, where $\Sigma _1 \times \set{0}$ is identified with $\Sigma _1 \times \set{{\bf d}}$ by an isometry.
\end{enumerate}

A property that obliges minimal submanifolds to intersect is now known as a Frankel property. Thus the space of closed minimal hypersurfaces in $\amb ^n$ with positive Ricci curvature satisfies the Frankel property and, nowadays, this is well understood if the ambient manifold has non-negative Ricci curvature and one of the minimal hypersurfaces is compact. When the minimal hypersurfaces are complete and non-compact we refer to \cite{JChoAFra18} for a detailed exposition and references therein

Hoffman and Meeks \cite{DHofWMee90} proved that a connected, proper, possibly branched, nonplanar minimal surface $\r ^3$ is not contained in a halfspace, their {\it Half-space Theorem}. They also proved that if $\Sigma _1$ and $\Sigma _2$ are two properly immersed minimal surfaces in the Euclidean three-space $\mathbb{R}^3$ such that $\partial \Sigma _1 = \partial \Sigma _2 = \emptyset$, then they are parallel planes; their {\it Strong Half-space Theorem}. To prove, they construct an orientable complete least area surface $L$ between $\Sigma _1$ and $\Sigma _2$; $L$ is stable hence a plane and they conclude with their Half-space Theorem.  Also, Meeks-Rosenberg proved the following Maximum Principle at Infinity:

\begin{quote}
{\bf Theorem \cite{WMeeHRos08}.} {\it Let $\Sigma _1$ be a properly immersed minimal surface in $\r ^3$. Then, if $\Sigma _2$ is a properly immersed minimal surface in $\r ^3 \setminus \Sigma _1 $, then $\Sigma _1$ and $ \Sigma _2$ are parallel planes}
\end{quote}

Many people have proved such Maximum Principles (see \cite{RLanHRos88,LMaz13,WMeeHRos08,ARosHRos10,HRosFSchJSpr13} and references therein), and the most general in ambient dimension three for us was done by L. Mazet \cite{LMaz13}; that we will explain in detail and generalize later for higher ambient dimension. 

Understanding stable minimal hypersurfaces is foundational on the subject. A complete two-sided stable minimal surface in $\r ^3$ is a flat plane \cite{DoCarCPen79,DFisRSch80,APog81}. R. Schoen \cite{RSch83} used this to prove a two-sided stable complete minimal surface $\Sigma$ in a complete $3-$manifold $\amb ^3$ of bounded geometry has bounded curvature, i.e., the sectional curvatures are bounded in absolute value and the injectivity radius has a positive lower bound. That is, there exists a constant $\Lambda >0$, depending on $\amb ^3$, such that if $\Sigma $ is as above, then 
$$ |A(x)|{\rm dist}_{\Sigma}(x , \partial \Sigma) \leq \Lambda  , \text{ for all } x \in \Sigma ,$$here $|A|$ is the norm of the second fundamental form of $\Sigma$. 

\begin{remark}
Schoen's curvature estimate depends also on an upper bound on the covariant derivate of the curvature tensor. H. Rosenberg, R. Souam and E. Toubiana \cite{HRosRSouETou10} obtained such a curvature estimate assuming only that the $3-$manifold $\amb ^3$ is of  bounded geometry.
\end{remark}

These results are, we believe, of the upmost importance for the study of complete minimal surfaces in $3-$manifolds. Recently, O. Chodosh and C. Li \cite{OChoCLi22} proved that an orientable, complete, stable minimal hypersurface in $\r ^4$ is a flat ($\r ^3$) hyperplane (see also \cite{GCatPMasARon22,OChoCLi23,OChoCLiPMinDStr24,HHonZYan24,LMaz24} for alternative proofs and extensions).

%Chodosh, Li and Stryker \cite{OChoCLiDStr22} used this to study stable minimal hypersurfaces in complete Riemannian $4-$manifolds $\amb ^4$ (resp. \cite{OChoCLiPMinDStr24} in dimension $5$ and \cite{LMaz24} in dimension $6$) of bounded geometry where they obtained curvature estimates for two-sided stable minimal hypersurfaces $\Sigma ^3 \subset \amb ^4$. 

Until recently, curvature estimates for stable minimal hypersurfaces in dimensions greater than three depend not only on the distance from the boundary but also on a volume growth hypothesis \cite{RSchLSim81,RSchLSimSYau75}. Thanks to recent results \cite{OChoCLiDStr22,OChoCLiPMinDStr24,LMaz24}, where they proved that two-sided stable minimal hypersurfaces in $\r ^n$, $n=4,5,6$, are hyperplanes, and a blow-up argument, we now have curvature estimates that only depend on the intrinsic distance to the boundary for twos-sided stable minimal hypersurfaces when the ambient manifold has bounded geometry and dimension $n=4,5,6$. It is expected that two-sided stable minimal hypersurfaces in $\r ^7$ are hyperplanes, which would imply curvature estimates for stable hypersurfaces that only depend on the intrinsic distance to the boundary in ambient manifolds of bounded geometry.

M. Anderson and L. Rodr\'iguez \cite{MAndLRod89} proved that the existence of an area-minimizing surface in a complete three-manifold of non-negative Ricci curvature splits the manifold. The idea behind \cite{MAndLRod89} is to construct a sequence of area-minimizing surfaces with two boundary components, one boundary on the original area-minimizing surface and the other boundary small and on a close equidistant surface, using the Douglas Criterion for the Plateau Problem. Then, taking limits when the boundary component on the original area-minimizing goes to infinity and the other component shrinks to a point on the equidistant, they are able to construct another area-minimizing smooth surface, passing through this point. Finally, they prove that this gives a local, and then global, foliation of the manifold. O. Chodosh, M. Eichmair and V. Moraru extended the above result under the weaker condition of non-negative scalar curvature of the ambient three-manifold (see \cite{OChoMEicVMor19} for a state of the art and references therein). In higher dimension, see also \cite{VMor16} when the area-minimizer is compact.

\subsection{Structure of the paper}

In Section \ref{Sect:Prel}, we include the basic notation and definitions we will use along the paper. An important tool in our proofs is to prove that area-minimizing ${\rm mod}(2)$ hypersurfaces have curvature estimates only depending on the ambient manifold and distance to the boundary when $n \leq 7$. This is done in detail in Section \ref{Sect:Mini}.

In Section \ref{Sect:MPI}, we first observe that Mazet's Theorem \cite{LMaz13} can be extended up to dimension seven with the same techniques, thanks to the Schoen-Simon-Yau's curvature estimates \cite{RSchLSim81,RSchLSimSYau75} and the regularity theory for area-minimizers developed by the combined works of De Giorgi, Federer, Fleming, Hardt and Simon (cf. \cite[Section 3]{NWic14}). Parabolicity is a necessary condition in dimension $n \geq 4$, even in Euclidean space, as the example of the catenoid shows. In fact, there are no parabolic, complete, minimal hypersurfaces by the monotonicity formula in $\mathbb{R}^n$, $n\geq 4$. Also, inspired by the Tubular Neighborhood Theorem of Meeks-Rosenberg in Euclidean three-space \cite{WMeeHRos08,HRos01}, we focus on the existence of an embedded $\epsilon -$tube when the ambient manifold has non-negative Ricci curvature. 

In Section \ref{Sect:Spli} we establish the two main new contributions of this paper. We restrict to $4-$manifolds of non-negative sectional curvatures and scalar curvature bounded below by a positive constant. Under these conditions, complete stable two-sided minimal hypersurfaces are totally geodesic and parabolic (cf. \cite{OChoCLiDStr22} and \cite[Theorem 1.2]{OMunJWan22}) which is a fundamental tool in the proof. The existence of a non-separating, properly embedded minimal hypersurface in such manifolds implies a local splitting. 

\begin{quote}
{\bf Theorem \ref{ThNonSeparating}} {\it Let $(\amb ^4, g)$ be a complete orientable $4-$manifold of bounded geometry, nonnegative sectional curvature, and scalar curvature bounded below by a positive constant. If $\Sigma \subset  \amb $ is a properly embedded orientable minimal hypersurface that is not separating, then $\Sigma$ is embedded and $\amb $ is a mapping torus over $\Sigma$.}
\end{quote}

This enables us to prove a Frankel-type property in these manifolds: 

\begin{quote}
{\bf Theorem \ref{CorNonSeparating}} {\it Let $(\amb ^4, g)$ be a complete orientable $4-$manifold of bounded geometry, nonnegative sectional curvature, and scalar curvature bounded below by a positive constant. Let $\Sigma _1, \Sigma _2 \subset  \amb$ be two properly embedded orientable minimal hypersurfaces. Then either $\Sigma _1$ intersects $\Sigma _2$, or $\Sigma _1$ and $\Sigma _2$ bound a region with a product metric.}
\end{quote}

\begin{remark}
The previous theorems can be extended to quasi-embedded minimal hypersurfaces, which will be defined later. In Section \ref{Sect:Spli}, we provide more detailed information about the splitting in the previous results. However, for the sake of clarity in exposition, we have chosen to present a simpler version here.
\end{remark}

In Section \ref{Sect:Para} we focus on area-minimizing hypersurfaces extending the Anderson-Rodr\'iguez Splitting Theorem \cite{MAndLRod89} up to dimension seven. We shall clarify that Anderson \cite{MAnd90} obtained a result like our Theorem \ref{ThAreaMinimizing} under the stronger assumption that the volume growth of geodesic balls of $\amb ^n$ is cubic.

\subsection{Concluding remarks}

An interesting subject for future research is the study of properly embedded minimal and constant mean curvature (CMC) hypersurfaces in \(\mathbb{S}^2 \times \mathbb{R}^2\) with the product metric, where \(\mathbb{S}^2\) is the standard 2-sphere of curvature one, and \(\mathbb{T}^2 \times \mathbb{R}^2\) with the product metric, where \(\mathbb{T}^2\) is the standard flat 2-torus. There are many examples of minimal hypersurfaces in these spaces. We know a great deal about properly embedded minimal surfaces \(\Sigma\) in \(\mathbb{S}^2 \times \mathbb{R}\) and \(\mathbb{T}^2 \times \mathbb{R}\) \cite{WMeeHRos98,WMeeHRos05/2,HRos02}. If \(\Sigma\) has bounded curvature, then \(\Sigma\) has linear area growth. When \(\Sigma\) has finite topology, \(\Sigma\) has bounded curvature. Hence, when \(\Sigma\) has finite topology, \(\Sigma \times \mathbb{R}\) has quadratic area growth in \(\mathbb{S}^2 \times \mathbb{R}^2\) or \(\mathbb{T}^2 \times \mathbb{R}^2\), thus it is parabolic.

Let \(l \subset \mathbb{R}^2\) be a line. Thus, \(\Sigma = \mathbb{S}^2 \times l\) or \(\Sigma = \mathbb{T}^2 \times l\) is a properly embedded parabolic minimal hypersurface to which the techniques of our theorems apply. For example, suppose \(\Sigma\) is a properly embedded minimal hypersurface of bounded curvature in \(\amb^4 := \mathbb{T}^2 \times \mathbb{R}^2\). If \(\Sigma\) does not separate \(\amb\), then \(\amb\) splits. This is proved in the spirit of our Theorem \ref{ThNonSeparating}. Cutting open \(\amb\) along \(\Sigma\) results in a manifold with two boundary components, each a copy of \(\Sigma\), which we call this new manifold $\anb$ with boundary $\partial \anb  = \mathcal S _1 \cup \mathcal S _2$. In $\anb$, we find a least-area minimal hypersurface \(F\) that is orientable and separates $\anb$ into two components, each containing a boundary component of $\anb$. Then, the universal cover of \(F\) is a flat hyperplane by Chodosh-Li's Theorem \cite{OChoCLi22}, therefore \(F\) is flat as well. Proceeding as in the proof of Theorem \ref{ThNonSeparating}, we conclude that $\anb$ is isometric to \(\Sigma \times [0, {\bf d}]\), where \( {\bf d} := \text{dist}_{\anb}(\mathcal{S}_1, \mathcal{S}_2)\). The Half-space theorem applies to \(\Sigma\). This convinced us that the study of properly embedded minimal (and CMC) hypersurfaces in these 4-manifolds is an interesting subject.

We also hope that the proof of Hadamard's Theorem given in the Appendix will generalize to dimension 3, to prove that an embedded complete geodesic in a closed $3-$manifold $\amb ^3$ of strictly positive sectional curvature must intersect any closed totally geodesic surface in $\amb ^3$.

To our knowledge, the Chodosh-Li-Stryker Theorem \cite{OChoCLiDStr22} on complete stable minimal hypersurfaces in positively curved $4-$manifolds has not been extended to either five or six-dimensional positively curved manifolds. Such a generalization would be essential to extend the results of this paper. Additionally, the techniques in this article should be adaptable to CMC hypersurfaces under appropriate conditions in the ambient manifold. Finally, the results that contain Section \ref{Sect:Spli} can be used to obtain topological obstructions to the existence of manifolds with non-negative Ricci curvature (see \cite{ZSheCSor08} and references therein). 

To conclude, regularity issues appear in Section \ref{Sect:Para} in dimensions higher than seven and, for us, it is not clear that our results extend straightforward.

\section{Notation and conventions}\label{Sect:Prel}

Let $(\amb , g)$ be a complete $n$-dimensional manifold, where $g$, or $\meta{\cdot}{\cdot}$, denotes the Riemannian metric on $\amb$. Let $\camb $ be the Levi-Civita connection on $\amb$. First, let us fix the notation. Set 
\begin{equation*}%\label{RiemannTensor}
\Rb (X,Y)Z := \camb _X \camb _ Y Z - \camb _Y \camb _X Z - \camb _{[X,Y]} Z , \, \, X,Y,Z \in \campo (\amb),
\end{equation*}
as the {\it Riemann Curvature Tensor}. Let $\set{e_i} \in \campo (U)$, $U \subset \amb$ open and connected, be a local orthonormal frame of the tangent bundle $TU \subset T \amb$. We denote by $\Rb _{ijkl} = \meta{\Rb(e_i , e_j)e_k}{e_l}$, and the {\it sectional curvatures} are given by $\Kb _{ij} := \meta{\Rb (e_i , e_j)e_j}{e_i} = \Rb _{ijji}$.

Moreover, we define the {\it Ricci curvature}, $\Ricb$, as the trace of the Riemann curvature tensor; and the {\it scalar curvature}, $\Sb$, as the trace of the Ricci tensor. Specifically:
$$ \Ricb (e_i , e_j) = \sum _{k=1}^n \Rb _{ikkj}  \text{ and  } \Sb (p) = \sum _{k=1}^n \Ricb _p(e_k ,e_k) .$$

We denote by ${\rm inj}(\amb)$ the infimum of the injectivity radius at any $p \in \amb$; where ${\rm inj}_\amb (p)$ is the largest $R > 0$ so that the exponential map ${\rm exp}_p : B (R) \subset T_p\amb \to \amb$ is a diffeomorphism onto its image. For $p \in \amb $ and $r>0$, we also denote by $\mathcal B_p (r)$ the metric ball in $\amb$ centered at the point $p \in \amb$ with radius $r>0$, $ \mathcal B_p (r) := \set{ q \in \amb \, : \, \, {\rm dist}_{\amb}(p,q) < r } $. Henceforth, $\mathcal H ^{l}$ denotes the $l$-dimensional Hausdorff measure (for the standard Riemannian measure).

\subsection{Parabolicity and quasi-isometries}

An important notion in this paper is parabolicity. A Riemannian manifold $\amb$ without boundary $\partial \amb = \emptyset$ (resp. with boundary $\partial \amb \neq \emptyset$) is {\it parabolic} (resp. {\it parabolic with boundary}) if every positive superharmonic function $u$ on $\amb$ (resp. with $u = 0$ on $\partial \amb$) must be constant. A crucial property of parabolicity is that it is preserved by quasi-isometries between Riemannian manifolds.

\begin{definition}\label{Def:kIso}
Let $(\amb ,g)$ and $(\mathcal N,h)$ be two $n$-dimensional Riemannian manifolds and
$\Phi : \amb \to \mathcal N $. Given $C \geq 1$, we say that $\Phi$ is a (local) quasi-isometry, or $C$-isometry, if for any $p \in \amb$ and $v \in T_p \amb$ it holds
$$ C^{-1} \, g(v,v) \leq h (d\Phi _p (v), d\Phi _p (v)) \leq C \, g(v,v). $$
\end{definition}

If $\amb$ and $\mathcal N$ have no boundary and $\Phi : \amb \to \mathcal N $ is a quasi-isometry diffeomorphism, $\amb$ is parabolic if and only if $\mathcal N$ is parabolic. Also, if $\partial \amb \neq \emptyset $, $\mathcal N$ parabolic, $\partial \mathcal N = \emptyset $  and $\Phi : \amb \to \mathcal N $ is a $k$-isometry injection, then $\amb$ is parabolic (with boundary). See  \cite[Section 3]{LMaz13} for details on parabolic manifolds and quasi-isometries.

\subsection{Hypersurfaces in manifolds}

Let $\Sigma ^{n-1}\subset \amb ^n$ be a two-sided (possibly with boundary $\partial \Sigma$) hypersurface immersed in a Riemannian manifold $\amb$. We still denote by $g$, or $\meta{\cdot}{\cdot}$, the Riemannian metric induced on $\Sigma$. Let $\camb$ and $\nabla $ be the Levi-Civita connections in $\amb$ and $\Sigma$ respectively. Let us denote by $\mathcal N \Sigma $ the unit normal bundle; that is, 
$$ \mathcal N \Sigma = \set{ (p,v) \, : \, v \in \mathcal N_p \Sigma = T_p \Sigma ^\perp \subset T_p \amb , \, |v| = 1} .$$

Since we are assuming that $\Sigma $ is two-sided, $\mathcal N \Sigma $ is trivial, and set $N : \Sigma \to \mathcal N \Sigma$ as a fixed unit normal. Therefore, by the Gauss Formula, we obtain
$$ \camb _X Y = \nabla _X Y +  \meta{A(X)}{Y}N \, \text{ for all } \, X, Y \in \campo (\Sigma),$$
where $A : T \Sigma \to T \Sigma$ is the Weingarten (or shape) operator, given by $ A(X) := - \camb _X N $. We denote the {\it mean curvature} and {\it extrinsic curvature}, respectively, as
$$ (n-1) H = {\rm Trace}(A) \, \, \,  \text{   and   } \, \, \, K_e = {\rm det}(A).$$

Let $\overline R$ and $R$ denote the Riemann Curvature tensors of $\amb$ and $\Sigma$, respectively. Then, the Gauss Equation states that for all $X,Y,Z,W \in \campo (\Sigma)$ we have
\begin{equation*}%\label{GaussEq}
\meta{R(X,Y)Z}{W}= \meta{\overline R (X,Y)Z}{W} + \meta{A(X)}{W} \meta{A(Y)}{Z} - \meta{A(X)}{Z}\meta{A(Y)}{W}.
\end{equation*}

In particular, if $K(X,Y)=\meta{R(X,Y)Y}{X}$ and $\Kb (X,Y) = \meta{\Rb (X,Y)Y}{X}$ denote the sectional curvatures in $\Sigma$ and $\amb$, respectively, of the plane generated by the orthonormal vectors $X,Y \in \campo (\Sigma)$, the Gauss Equation becomes 
\begin{equation*}%\label{GaussSectionalEq}
K(X,Y) = \Kb (X,Y) + \meta{A(X)}{X}\meta{A(Y)}{Y} - \meta{A(X)}{Y}^2 .
\end{equation*}

We also denote by $Ric$ and $Sc$ the Ricci tensor and scalar curvature of $\Sigma$ with the induced metric. For $p \in \Sigma $ and $r>0$, let $\mathcal D_p (r)$ be the metric ball in $\Sigma$ centered at the point $p \in \Sigma$ with radius $r>0$, $ \mathcal D_p (r) := \set{ q \in \Sigma \, : \, \, {\rm dist}_{\Sigma}(p,q) < r }$. 

A special class of hypersurfaces we will use in this paper is the class of properly quasi-embedded, which we shall define next:
\begin{definition}\label{QuasiEmbedded}
Let $\Sigma \subset \amb$ be a properly immersed, two-sided hypersurface in a complete manifold $\amb$. $\Sigma$ is said to be {\it properly quasi-embedded} if there exists a compact set $\mathcal K \subset \amb$ such that $\Sigma \setminus \mathcal K$ is embedded. 
\end{definition}

The above definition leads us to:

\begin{definition}\label{DefNonSeparating}
Let $\Sigma \subset \amb$ be a properly quasi-embedded hypersurface. We say that $\Sigma$ is {\it non-separating} if for $p \in \Sigma \setminus \mathcal K$, the embedded part of $\Sigma$, there exists a simple closed curve $\gamma \subset \amb$ intersecting $\Sigma$ transversally at the single point $p$. 
\end{definition}

\subsection{Tubes}

Let $\amb$ be a complete manifold and $\Sigma \subset \amb$ be a two-sided embedded hypersurface. An $\epsilon-$tubular neighborhood of $\Sigma$ (in short, {\it$\epsilon$-tube}) is the set of points in $\amb$ at distance at most $\epsilon$, that is,
$$ {\rm Tub}_\epsilon (\Sigma) = \set{ p \in \amb \, : \, {\rm dist}_{\amb}(p, \Sigma) \leq \epsilon } .$$

Moreover, 
\begin{definition}\label{ETube}
We say that ${\rm Tub}_\epsilon (\Sigma)$ is:
\begin{enumerate}
\item[(a)] {\bf $\epsilon$-Embedded:} The map $\Psi: \Sigma \times [-\epsilon , \epsilon ] \to {\rm Tub}_{\epsilon}(\Sigma )$ given by $\Psi (p,t) := {\rm exp}_p (t N(p))$ is an embedding. Here, ${\rm exp} $ is the exponential map of $\amb$ and $N$ is a unit normal along $\Sigma$.

\item[(b)] {\bf Well-oriented:} Each equidistant hypersurface
$$ \Sigma (\epsilon ’) = \set{p \in \amb \, : \, {\rm dist}_{\amb} (p, \Sigma) = \epsilon ’},$$ at distance $\epsilon ’ \in [-\epsilon , \epsilon ] \setminus \set{0}$, is mean convex (non-negative mean curvature) with the orientation that points away from $\Sigma$.

\item[(c)] {\bf Bounded extrinsic geometry:} There exists a constant $C \geq 1$ so that $\Sigma (\epsilon ’)$ is $C$-isometric to $\Sigma$ with second fundamental form uniformly bounded by $C$ for all $\epsilon ’ \in [-\epsilon , \epsilon ]$.
\end{enumerate}
\end{definition}

\subsection{Minimal hypersurfaces}

Let $\Sigma \subset \amb $ be a compact (two-sided) hypersurface with boundary $\partial \Sigma $ (possibly empty), $N$ a unit normal to $\Sigma$, and $0 \in \mathcal{I} \subset \r$ an interval. Let $\Psi : \Sigma\times \mathcal{I} \to \amb$ be an immersion defining a deformation of $\Sigma (0) = \Psi\left(\Sigma \times \set{0}\right)$. Let $X(p) = \derb{}{t}{t=0}\Psi (p,t)$ denote the variational vector field of this variation, then:

\begin{theorem}[First variation Formula]
In the above conditions, if $\eta$ denotes the inward unit conormal along $\partial \Sigma$ and $\vec{H}$ the mean curvature vector along $\Sigma$, then

\begin{equation}\label{1VF}
\derb{}{t}{t=0}{\rm Area}\left(\Sigma (t)\right) = \int _{\partial \Sigma} \meta{X}{\eta} ds - 2 \int _{\Sigma}\meta{X}{\vec{H}} dv_g .
\end{equation}
\end{theorem}

From the First Variation Formula \eqref{1VF}, $\Sigma $ is minimal if, and only if, it is a {\it critical point} of the area functional for any compactly supported variation fixing the boundary. 

We continue this section by recalling the second variation formula of the area for minimal hypersurfaces. We focus on especially interesting variational vector fields $X$, those which are normal, i.e., $X(p)= f(p) N (p)$ and $f \in C^\infty _0 (\Sigma)$, where $C^\infty _0 (\Sigma)$ denotes the linear space of piecewise smooth functions compactly supported on $\Sigma$ that vanish at the boundary $\partial \Sigma$, i.e., $f_{|\partial \Sigma} \equiv 0$.

\begin{remark}
In fact, we only need $f \in H^{1,2}_0 (\Sigma)$, where $H^{1,2}_0 (\Sigma)$ is the closure of $C^\infty _0 (\Sigma)$ in the topology associated to the Sobolev norm. If $\Sigma $ is not compact, we will consider $f$ as a compactly supported function. 
\end{remark}

A minimal hypersurface is said {\it stable} if for any relatively compact domain $\Omega \subset \Sigma$, the area of $\Omega$ cannot be decreased up to second order by a variation $\Omega (t)$ of the domain leaving the boundary $\partial \Omega (t)$ fixed. In other words, if 
$$\left.\frac{d^2}{dt^2}\right\vert _{t=0} {\rm Area}(\Sigma (t)) \geq 0 , $$for any variation of the domain leaving the boundary fixed. We can write the Second Variation Formula as
\begin{equation}\label{2VF}
\left.\frac{d^2}{dt^2}\right\vert _{t=0} {\rm Area}(\Sigma (t))= - \int _{\Sigma} f L f \, dv_g ,
\end{equation}where $L : C_0 ^\infty (\Sigma) \to C_0 ^\infty (\Sigma) $ is the linearized operator of mean curvature or {\it Jacobi operator} (see \cite{HRos93}), that is,
\begin{equation}\label{1VH}
L f :=\left.\frac{d}{dt}\right\vert _{t=0} H(t) = \Delta f+ (|A|^2 + \Ricb  (N, N ))f,
\end{equation}where $H(t)$ is the mean curvature of the hypersurface $\Sigma (t)$ and  $|A|^2$ denotes the square of the length of the second fundamental form of $\Sigma$. The above formula \eqref{1VH} is also referred to as the {\it First Variation Formula for the Mean Curvature} and it holds whether $\Sigma$ is minimal or not. It is well-known (cf. \cite{PLiJWan04}) that:

\begin{proposition}\label{PropParabolicStable}
Assume that $(\amb ^n ,g)$ has nonnegative Ricci curvature $\Ricb \geq 0$ and $\Sigma \subset \amb$ is a complete, two-sided, parabolic, stable minimal hypersurface. Then, $\Sigma$ is totally geodesic and the Ricci curvature of the ambient manifold vanishes along $\Sigma$ in the normal direction, i.e., $\Ricb (N,N) \equiv 0$ along $\Sigma$.
\end{proposition}

\subsection{Local foliations}

The following lemma is well known in this subject and we welcome a reader to tell us to whom it should be attributed.

\begin{lemma}\label{Foliation}
Let $\Sigma^{n-1}$ be a complete, embedded, two-sided minimal hypersurface in a complete manifold $\amb^n$ of non-negative Ricci curvature. Suppose that for some $\delta >0$ the map $\Psi: \Sigma \times [0, \delta] \to \amb $, $\Psi(p,t) := {\rm exp}_p (t N(p))$, where $N$ is a unitary normal to $\Sigma = \Psi(\Sigma \times \set{0})$, is an embedding.

If there is a $\bar t \in (0, \delta]$ such that the equidistant $\Sigma (\bar t) = \Psi(\Sigma \times \set{\bar t})$ is a minimal hypersurface, then $\Psi: \Sigma \times [0, \bar t] \to \amb$ is an isometry onto its image, where $\Sigma \times [0, \bar t]$ has the product metric $\sigma + dt^2$; $\sigma$ the induced metric on $\Sigma$.
\end{lemma}
\begin{proof}
The geodesics $\gamma _p (t) = F(p,t)$, $t \in [0,\delta]$, are minimizing, $\Psi$ has no critical points so there are no conjugate points on $\gamma _p$ and no two geodesics $\gamma _p$, $\gamma _q$ intersect if $p \neq q$. 

Let $d = {\rm dist}_{\amb}\left(\Psi(p,t), \Psi(\Sigma \times \set{0})\right)= t$ denote the distance function on $\Psi(\Sigma \times [0, \delta]) \subset \amb$ to $\Sigma = \Psi(\Sigma \times \set{0})$; so $\camb d = N_t $ on $\Psi(\Sigma \times [0, \delta])$ where $N_t$ is the unitary normal along the equidistant $\Sigma (t) =\Psi(\Sigma \times \set{t})$, $t \in [0,\delta]$, such that $N_0 =N$. 

The First Variation Formula for the Mean Curvature \eqref{1VH} of the equidistant hypersurfaces $\Sigma (t)$, $t \in [0, \delta]$, is:
$$ H' (p,t) = |A_t (p)|^2 + \Ricb(\camb d _{\Psi(p,t)}, \camb d _{\Psi(p,t)}) \geq 0 ,$$where $A_t (p)$ is the second fundamental of $\Sigma (t)$ at $\gamma _p (t)$; and the sign of $(n-1)H(p,t) = {\rm Trace}\left( A_t (p) \right)$ is such that 
$$ A_t (p) (u) = - \camb _{u} \left( \camb d _{\Psi(p,t)}\right) , \, \, u \in T_{\Psi(p,t)}\Sigma (t).$$

Since $H(p, 0)=H(p, \bar t) =0$ for all $p \in \Sigma$ and $H$ is non-decreasing on $t$, then $H(p,t) = 0$ for all $(p,t) \in \Sigma \times [0, \bar t]$, each equidistant $\Sigma (t)$, $t \in [0, \bar t]$ is totally geodesic ($A_t \equiv 0$ for all $t \in[0,\bar t]$), and $\Ricb(\camb d , \camb d ) \equiv 0 $ as well in this region $ \Psi(\Sigma \times [0, \bar t])$.

Observe that $\camb d $ is parallel in $\amb$; $\camb _v \camb d =0$ for all $v \in T_{\Psi(p,t)} \amb$. To see this, decompose a tangent vector $v \in T_{\Psi(p,t)} \amb$ into its part tangent to the equidistant $\Sigma (t)$ and normal to $\Sigma (t)$; i.e., parallel to $\camb d$. Since $A_t \equiv 0$, $\camb _v \camb d =0 $ for all $v$ tangent to $\Sigma (t)$ and $\camb _{\camb d} \camb d = 0$, so $\camb _v \camb d =0$ for all $v \in T_{\Psi(p,t)} \amb$. 

Next, notice that $\camb d$ is a Killing field in $\amb$: 
$$ \meta{\camb _X \camb d}{Y} + \meta{\camb _Y \camb d }{X} = 0 \text{ for any } X,Y \in T_{\Psi(p,t)} \amb .$$

To check the above equation, again decompose $X$ and $Y$ into their tangent and normal parts. Hence, the map $\Sigma (t) \to \Sigma (s)$ given by the integral curves $\gamma _p (t)$ of $\camb d$ is an isometry for $t,s \in[0, \bar t]$. Thus $\Psi$ preserves the scalar product of the metric $\sigma + dt^2$, which proves the lemma.
\end{proof}

\section{Area-minimizing hypersurfaces}\label{Sect:Mini}

In this section, we use several advanced results in Geometric Measure Theory (GMT) to address the problems we consider. We will not delve into the techniques involved in proving these theorems but will provide exact references when necessary. We are particularly grateful to Brian White for explaining several important subtleties of GMT and for patiently answering our questions. His lectures in 2012 on GMT \cite{White}, written by Otis Chodosh, are an inspiring presentation of GMT, and we will frequently reference this text (cf. also \cite{CDeLJHirAMarSStu20, FMor16,NWic14}). In this section, we assume that the ambient manifold $\amb^n$, $3 \leq n \leq 7$, is connected (possibly with boundary), orientable, and of bounded geometry.

\begin{definition}
Let $\mathcal F \subset \amb$ be an orientable properly embedded minimal hypersurface. We say that $\mathcal F$ is area-minimizing with $\zz _2$-coefficients, in short, area-minimizing ${\rm mod}(2)$, if for any compact domain $C \subset \mathcal F$ with $C^1$ boundary and $\mathcal S \subset \amb$ a compact hypersurface with $\partial \mathcal S = \partial C$ (not as oriented hypersurfaces) and homologous to $C$, $\mathcal H ^{n-1}( C) \leq \mathcal H ^{n-1}(\mathcal S)$.
\end{definition}

We establish here the result from GMT we will use:

\begin{theorem}\label{ThCurrent}
Let $(\amb^n, g)$, $3 \leq n \leq 7$, be a complete, orientable Riemannian manifold and $\mathcal C \subset \amb$ be a relatively compact domain with piecewise, mean convex, $C^1$ boundary. Let $W \subset \partial \mathcal C$ be a compact domain such that $\partial W := \Gamma$ is a disjoint union of embedded hypersurfaces in $\partial \mathcal C$. Then, there exists a compact area-minimizing ${\rm mod}(2)$ hypersurface $\mathcal F \subset \mathcal C$, $\partial \mathcal F = \Gamma$ and $\mathcal F$ homologous to $W$, i.e., $\mathcal F$ is area-minimizing ${\rm mod}(2)$ relative to $(\mathcal C, W, \Gamma)$.
\end{theorem}

The strategy to prove Theorem \ref{ThCurrent} is as follows. First, observe that the Compactness Theorem \cite[Section 3.4]{White} guarantees a convergent subsequence $F_i$ bounding $\Gamma$ whose masses approach the infimum of the mass of integral currents spanning $\Gamma$. Such a subsequence converges to a minimizer $\mathcal F$ \cite[Theorem 8.1]{White} being an integral current. Regularity theory ${\rm mod}(2)$ (cf. Theorem \cite[Theorem 9.14]{White}) implies that $\mathcal F$ is smooth, embedded, and two-sided for $3 \leq n \leq 7$.

\subsection{Construction of area-minimizers}

Let $(\amb^n, g)$ be an oriented complete manifold with boundary. Moreover, its boundary $\partial \amb$ is said to be {\it a good barrier} if it is the disjoint union of smooth hypersurfaces meeting at interior angles less than or equal to $\pi$ along their boundaries and such that the mean curvature of these hypersurfaces with respect to the inward orientation is non-negative.

\begin{lemma}\label{LemF}
Let $(\amb^n, g)$, $3 \leq n \leq 7$, be a connected, orientable, complete manifold of bounded geometry with boundary $\partial \amb$ which is a good barrier and has two connected components, $\mathcal S _1$ and $\mathcal S_2$. Let $\gamma \subset \amb$ be a geodesic joining a point $p_1 \in \mathcal S_1$ to a point $p_2 \in \mathcal S_2$. Assume there exists an exhaustion by compact domains $\set{C_i}_{i \in \np}$, $C_i \subset C_{i+1}$ for all $i \in \n$, whose boundaries are good barriers. Then, $\amb$ contains a properly embedded, orientable, stable minimal hypersurface $\mathcal F$. In fact, $\mathcal F$ is area-minimizing ${\rm mod}(2)$, intersects $\gamma$, and is homologous to one of the boundary components of $\amb$, say $\mathcal S_1$.
\end{lemma}

Lemma \ref{LemF} essentially follows from the arguments in the proof of \cite[Lemma 2.2]{MAnd90}, with the modification that the ambient manifold is complete without boundary, while the boundary of the ambient manifold acts as a {\it good barrier}. However, we include a detailed proof here for the reader's convenience. Additionally, this proof provides the technique to demonstrate a Tubular Neighborhood Theorem for area-minimizers ${\rm mod}(2)$ (cf. Theorem \ref{Mod2Proper} below).

\begin{proof}[Proof of Lemma \ref{LemF}]
Without loss of generality, we can assume that $\gamma : [0,1] \to C_i$, $\gamma (0) = p_1$ and $\gamma (1) = p_2$, for all $i \in \mathbb{N}$. Denote by $W_i$ the connected component of $C_i \cap \mathcal S_1$ that contains $p_1 \in W_i$; set $\Gamma_i = \partial W_i$. Since $\partial C_i$ leaves any compact set of $\amb$ for $i$ large, ${\rm dist}_{\amb}(\gamma, \Gamma_i) \to + \infty$ as $i \to + \infty$.

Fix $j \in \mathbb{N}$ and let $i > j$. Let $F_{i} \subset C_{i}$ be an embedded hypersurface in $\amb$ such that $\partial F_i = \Gamma_i$ and $F_i$ minimizes area ${\rm mod}(2)$ given by Theorem \ref{ThCurrent}. Orient $F_i$ by the normal pointing to the domain bounded by $F_i \cup W_i$ (cf. \cite[Section 5.7, page 66]{FMor16}). Let $F_{ij}$ be the component of $F_i \cap C_j$ such that $F_{ij} \cap \gamma \neq \emptyset$; this component exists since the intersection number of $W_j$ and $\gamma$ is one ($p_1 \in W_j$). 

Let $\bar t _{ij} := {\rm min} \set{ t \in [0,1] \, : \, \, \gamma (t) \in F_{ij} } \geq 0$ and 
 $\gamma (\bar t _{ij}):=x_{ij} \in F_{ij} \cap \gamma$ be the highest point of $F_{ij} \cap \gamma$, so the arc of $\gamma$ joining $x_{ij}$ to $p_1$ is contained in the domain bounded by $F_i \cup W_i$, and the normal to $F_{ij}$ at $x_{ij}$ points into this domain (cf. Figure \ref{Figure1}). 

\begin{figure}[htbp]
\begin{center}
\includegraphics[scale=0.2]{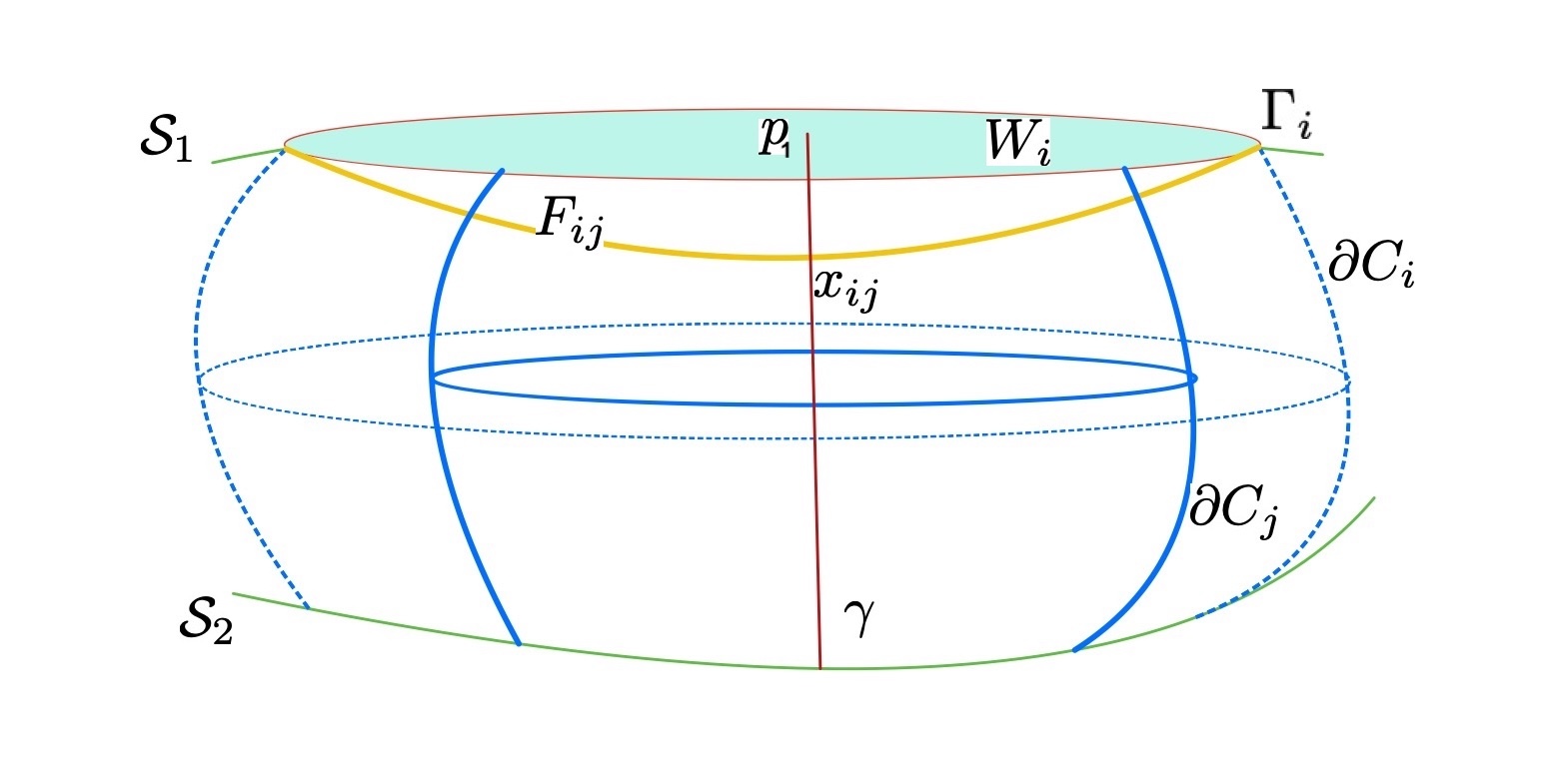}
\caption{Construction of $F_{ij}$}
\label{Figure1}
\end{center}
\end{figure}

We now wish to show the curvature of $F_{ij}$ is bounded independently of $i$, and to prove this we will show they have good area bounds to which the curvature estimates of Schoen-Simon-Yau \cite{RSchLSim81, RSchLSimSYau75} apply. Cover $C_j$ by {\it small} metric balls $\mathcal B_{x_k} (2 \delta)$ such that the balls $\mathcal B_{x_k} (\delta)$ already cover $C_j$. We choose $\delta$ sufficiently small so that each $\mathcal B_{x_k} (2\delta)$ is strictly convex and diffeomorphic to the unit ball of $\r^n$. Since the extrinsic ball $\mathcal B_{x_k} (2\delta)$ is diffeomorphic to a ball in $\r^n$, every embedded, piecewise smooth hypersurface in $\mathcal B_{x_k} (2\delta)$ bounds a domain in $\mathcal B_{x_k} (2\delta)$. $C_j$ is compact and $\amb$ has bounded geometry, so a finite number of $\mathcal B_{x_k} (\delta)$ suffice to cover $C_j$ and their number depends on the volume of $C_j$ and curvature bounds of $\amb$. The balls are sufficiently small so the volume of their boundaries is at most $c \delta^{n-1}$ for some constant $c$ independent of $k$. For later use, we will cover $\amb$ with a countable number of such {\it small balls}.

\begin{quote}
{\bf Claim A:} {\it The volume of $F_{ij} \cap \mathcal B_{x_k} (2\delta)$ is bounded above by half the volume of $\partial \mathcal B_{x_k} (2\delta)$.}
\end{quote}
\begin{proof}[Proof of Claim A]

This statement can be found in \cite[Example 8.9]{White}, however, we prefer to include it for the reader’s convenience.

We will first suppose $\mathcal B_{x_k}(2\delta) \equiv \mathcal B$ is the unit ball of $\r^3$ and $F_{ij} \cap \mathcal B = S_1 \cup \ldots  \cup S_m$ are pairwise disjoint embedded surfaces, each $S_l$ meeting $\partial \mathcal B$ transversally in a finite number of (topological) circles. Let $S := S_1 \cup \ldots \cup S_m$ and $\beta_1, \ldots, \beta_r$ be the pairwise disjoint circles of $S \cap \partial \mathcal B$. The complementary components of  $\partial \mathcal B \setminus \bigcup _{a=1}^r \beta _a$ can be {\it painted with two colors}, which we will denote by $\cR$ and $\cB$. So if $D$ is one such component labeled $\cR$, then each circle in $\partial D$ is two-sided, with the label $\cR$ on the $D$ side, and the label $\cB$ on the other side (a closed embedded hypersurface of $\partial \mathcal B$ separates it into two components and is orientable). Now label one component $\cR$, then label the components of the other side of each circle in its boundary by $\cB$, and continue labeling until each component is assigned $\cR$ or $\cB$. So each circle $\beta _a$ is labeled $\cR$ on one side and $\cB$ on the other side.

The reader can convince themselves that this labeling works as follows. Choose some circle $\beta_1$ that bounds a domain $D_1$ on $\partial \mathcal B$ that contains no other circle. Label $D_1$ by $\cR$. Let $D_2$ be the connected component of $\partial \mathcal B \setminus \left( D_1 \cup \beta_2 \cup \ldots \cup \beta_r \right)$ containing $\beta_1$ in its boundary and let $\beta_2, \ldots, \beta_{\bar r}$, $\bar r \leq r$ and (relabeling if necessary) the remaining circles of $\partial D_2$. Label $D_2$ by $\cB$. Each $\beta_2, \ldots, \beta_{\bar r}$ bounds a disk in $\partial \mathcal B \setminus D_1$, they are pairwise disjoint, and we continue the labeling in each disk bounded by the $\beta_2, \ldots, \beta_{\bar r}$. So the labeling is an increasing sequence of domains of $\partial \mathcal B$ and terminates after a finite number of steps.

Now, $S$ is the disjoint union of the $S_1, \ldots, S_l$. We attach the $\cR$ domains $D_\alpha$ of $\partial \mathcal B$ to $S$ by gluing each circle in $\partial D_\alpha$ to the circle from which it came in $S$; i.e., in $\partial S_\alpha$ for some $S_\alpha$. This gives a finite number of piecewise smooth closed embedded surfaces $M_1, \ldots, M_{r’}$, $r’ \leq r$, in $\mathcal B$; smooth except along the circles. Each $M_\nu$ bounds a compact domain in $\mathcal B$ and the domains are pairwise disjoint. Each $M_\nu \cap {\rm int}(\mathcal B)$ is homologous ${\rm mod}(2)$ to the $\cR$ domains in $\partial M_\nu$.

Since $S$ is area-minimizing ${\rm mod}(2)$, the area of $M_\nu \cap {\rm int}(\mathcal B)$ is at most the area of the $\cR$ or $\cB$ domains in $\partial M_\nu$. Thus, the area of $S$ is less than half the area of $\partial \mathcal B$.

The reader can easily check that the proof works as well with $\mathcal B$ replaced by $\mathcal B_{x_k} (2\delta)$ and the $S_l$’s replaced by the connected components of $F_{ij} \cap \mathcal B_{x_k} (2\delta)$, each connected component meeting $\partial \mathcal B_{x_k} (2\delta)$ transversally in pairwise disjoint closed embedded hypersurfaces in $\mathcal B_{x_k} (2\delta)$ (playing the role of the circles), which proves Claim A.
\end{proof}

Then, we can assert:

\begin{quote}
{\bf Claim B:} {\it Schoen-Simon-Yau's curvature estimates \cite{RSchLSim81, RSchLSimSYau75} apply, i.e., $F_{ij}$ has bounded curvature independent of $i > j$.}
\end{quote}
\begin{proof}[Proof of Claim B]
Since $C_j$ is covered by a finite number of {\it small} balls $\mathcal B_{x_k}(2\delta)$, with $\mathcal B_{x_k}(\delta)$ already covering $C_j$, and in each such ball
$$\mathcal H ^{n-1}\left( F_{ij} \cap \mathcal B_{x_k}(2\delta) \right) \leq c \delta ^{n-1},$$ where $c$ depends on the geometry of $\amb$ and $c$ is the same for all $\delta$ sufficiently small and for all $i,j,k$. In fact, the bounded geometry of $\amb$ will give the same estimate in each of the countable $\mathcal B_{x_k}(2\delta)$ of the covering of $\amb$ for any properly embedded area-minimizing hypersurface ${\rm mod}(2)$ in $\amb$ by Claim A. Thus, for any such hypersurface $\mathcal F$, one has (cf. \cite{RSchLSim81, RSchLSimSYau75})
$$ {\rm sup}_{\mathcal F \cap \mathcal B_{x_k}(\delta)} |A_{\mathcal F}|^2 \leq  C(n,c) \delta ^{-2},$$
for some constant $C(n,c)$ that only depends on the dimension $n \leq 7$ and $c$ above. Therefore, $F_{ij}$ has bounded curvature independent of $i$ greater than $j$ and Claim B is proved.
\end{proof}

The $F_{ij}$ have bounded curvature by Claim B, thus they are (locally) graphs of bounded geometry over some ball of radius $\tau>0$ in their tangent hyperplane at each point and $\tau$ only depends on the curvature bounds. More precisely, the intrinsic ball of radius $\tau$ centered at $p \in F_{ij}$, $\mathcal D^{ij}_{p}(\tau)$, is the graph (in normal geodesic coordinates) of a function $u^{ij}_p$ defined on ${\rm exp}(D^{ij}_p)$, where $D^{ij}_p \subset T_p F_{ij}$ is an open neighborhood of ${\bf 0} \in T_p F_{ij}$ in $T_p F_{ij}$, such that
$$ u^{ij}_p ({\bf 0}) = 0 = |\nabla u^{ij}_p ({\bf 0})| $$and
$$ |\nabla u^{ij}_p (y)|^2 \leq  \Lambda \tau^2 \text{ and } |\nabla^2 u^{ij}_p (y)|^2 \leq \Lambda \tau^2  \text{ for all } y \in D^{ij}_p ;$$
which implies that $\partial \mathcal D^{ij}_p (\tau) \cap \mathcal B _p (\tau /2) = \emptyset$. Since $\partial \mathcal D^{ij}_p (\tau) \cap \mathcal B _p (\tau /2) = \emptyset$, monotonicity of the area ratios gives a strictly positive lower bound for the area for all these local graphs $\mathcal D^{ij}_p (\tau)$ for all $p \in F_{ij}$ and $i > j$.

An important remark is that if $q \in F_{ij} \setminus \mathcal D^{ij}_p (2\tau)$, then $q$ cannot be too close (in the extrinsic distance) to $p$. For when $q$ is close to $p$, $\mathcal D^{ij}_p(\tau)$ and $\mathcal D^{ij}_q (\tau)$ are disjoint since $F_{ij}$ is embedded. Then, by compactness of minimal graphs, $\mathcal D^{ij}_p (\tau /2)$ and $\mathcal D^{ij}_q (\tau /2)$, if $p$ and $q$ are sufficiently close and $\tau $ sufficiently small, are graphs over the same domain ${\rm exp}(U_p)$, $U_p \subset T_p F_{ij}$, that are $C^2$-close. Hence $\partial \mathcal D^{ij}_p (\tau /2)$ and $\partial \mathcal D^{ij}_q (\tau /2)$ are $C^2$-close and $\partial \mathcal D^{ij}_p (\tau /2) \cup \partial \mathcal D^{ij}_q (\tau /2)$ is the boundary of a (annular) hypersurface $\mathcal A$ of small volume. Replace $F_{ij}$ by removing $\mathcal D^{ij}_p (\tau /2) \cup \mathcal D^{ij}_q (\tau /2)$ and adding $\mathcal A$. This is congruent to $F_{ij} \,  {\rm mod}(2)$ and of less volume than $F_{ij}$ if $q$ is sufficiently close to $p$. This is impossible since $F_{ij}$ is area-minimizing ${\rm mod}(2)$.

Now, $F_{ij}$ has finite volume independent of $i$ so can be covered by a finite number of the $\mathcal D^{ij}_p(\tau)$. Next, we construct a limit $\mathcal F_j$ of a subsequence of $F_{ij}$ which will be an oriented embedded area-minimizing ${\rm mod}(2)$ hypersurface and $\mathcal F_j \cap \gamma \neq \emptyset$. Let $i \to \infty$ and choose a subsequence of the $x_{ij}$ that converges to a point $x_j \in \gamma$ and also their tangent hyperplanes at the points of the subsequence converge to a hyperplane $P_j \subset T_{x_j}\amb$. The $\tau$-graphs $\mathcal D_{x_{ij}} (\tau ) \subset F_{ij}$ of the subsequence converge to a minimal $(\tau/2)$-graph $\mathcal G_j$ (in normal geodesic coordinates over $P_j$) at $x_j$.

Since the tubular neighborhoods of the $F_{ij}$ have a fixed size independent of $i$, and the $x_{ij}$ are the highest points of intersection of $F_{ij}$ with $\gamma$, the convergence of the $\mathcal D_{x_{ij}} (\tau)$ to $\mathcal G_j$ is multiplicity one. Recall that $F_i$ and $W_i$ are homologous, hence they bound a domain $\Omega_i \subset \amb$. Along the converging subsequence (also called $F_{ij}$), the normal points into $\Omega_i$. So, the normal of $F_{ij}$ at $x_{ij}$ pointing towards $\Omega_i$, defines an orientation of the limit surface $\mathcal G_j$. 

Each point of $\partial \mathcal G_j$ is also a limit of points in the graphs $\mathcal D_{x_{ij}} (\tau)$ and the normals to $\mathcal D_{x_{ij}} (\tau)$ converge to the normal to $\mathcal G_j$ at each boundary point of $\mathcal G_j$. The same limiting process at the boundary points of $\mathcal G_j$ extends $\tau/2$ further from its boundary. Continuing this by extending the boundary, we obtain an embedded minimal hypersurface $\mathcal F_j$ (with boundary) in $C_j$; which is a limit of the components of the $F_{ij}$ of the subsequence in $C_j$ as $i \to \infty$. The sequence converges $C^2$ to $\mathcal F_j$ ($\mathcal F_j$ compact) so $\mathcal F_j$ together with a compact domain on $\partial C_j$ is homologous to $C_j \cap \mathcal S_1$ (cf. \cite[Lemma 1.1]{WMeeHRos05} for a detailed proof in dimension three).

Now repeat this process starting with the above subsequence of the $F_i$’s and look at the components intersecting $C_{2j}$ (assume $i > 2j$) that equal the components in $C_j$ we made limit to $P_j \subset T_{x_j}\amb$. A subsequence of this family converges to a limit $\mathcal F_{2j}$ that equals $\mathcal F_{j}$ in $C_j$; we have the same local area and curvature estimates in the balls $\mathcal B_{x_k} (2\delta)$ covering $C_{2j}$. Then, letting $j \to \infty$, we obtain the complete, embedded, orientable, area-minimizing ${\rm mod}(2)$ hypersurface $\mathcal F \subset \amb$. Moreover, Since $\mathcal F \cap C_j$ is homologous to $C_j \cap \mathcal S_1$ for all $j \in \mathbb{N}$, then the limit hypersurface $\mathcal F$ is homologous to $\mathcal S_1$.
\end{proof}

We showed above that if $q \in F_{ij} \setminus \mathcal D_p (2\tau)$, then $q$ cannot be too close (in the extrinsic distance) to $p$. This argument also proves:

\begin{theorem}[Tubular Neighborhood Theorem for ${\rm mod}(2)$-minimizers]\label{Mod2Proper}
Let $(\amb^n, g)$, $3 \leq n \leq 7$, be a complete, orientable manifold of bounded geometry. If $\mathcal F \subset \amb$ is a complete, orientable, embedded hypersurface that is area-minimizing ${\rm mod}(2)$, then $\mathcal F$ is proper. Moreover, there exists $\epsilon >0$ (depending only on the bounds for the sectional curvatures and injectivity radius of $\amb$ and the second fundamental form of $\mathcal F$) such that ${\rm Tub}_\epsilon (\mathcal F)$ is embedded.
\end{theorem}
\begin{proof}
By Claims A and B in Lemma \ref{LemF}, $\mathcal F$ has a bounded second fundamental form. Hence, there exists $\tau >0$ (depending on the curvature bounds) such that the intrinsic metric ball of radius $\tau$ centered at $p \in \mathcal F$, $\mathcal D_{p}(\tau)$, is the graph (in normal geodesic coordinates in $\amb$) of a function $u_p$ defined on ${\rm exp}(D_p)$, $D_p \subset T_p \mathcal F$ an open neighborhood ${\bf 0} \in T_p \mathcal F$ in $T_p \mathcal F$, such that $\partial \mathcal D _p (\tau) \cap \mathcal B _p (\tau /2) = \emptyset$, here $\mathcal B _p (\tau /2)$ is the metric ball of radius $\tau/2$ centered at $p$ in $\amb$. The result will be proven if we show:

\begin{quote}
{\bf Claim A:} {\it There exists $\epsilon >0$ (depending on $\tau$) such that for any two points $p,q \in \mathcal F$ with ${\rm dist}{\mathcal F}(p,q) > 2\tau$, then ${\rm dist}_{\amb}(p,q) > \epsilon$.}
\end{quote}
\begin{proof}[Proof of Claim A]
By contradiction, assume there exists $p \in \mathcal F$ and a sequence $\set{p_n}_{n \in \np} \subset \mathcal F$ such that ${\rm dist}_{\mathcal F}(p,p_n) > 2\tau$ and ${\rm dist}_{\amb}(p,p_n) < 1/n$. By compactness of minimal graphs, $\mathcal D_{p_n} (\tau /2)$ and $\mathcal D_p (\tau /2)$ are graphs over the same domain ${\rm exp}(U_p)$, $U_p \subset T_p \mathcal F$ for $n$ large enough, that are $C^2$-close. Also $\mathcal D_{p_n} (\tau /2) \cap \mathcal B _p (\tau /4)$ and $\mathcal D_p (\tau /2) \cap \mathcal B _p (\tau /4)$ are $C^2$-close and they are the boundary of a (annular) hypersurface $\mathcal A \subset \partial \mathcal B _p (\tau /4)$ of small volume. Replace $\mathcal F$ by removing $\left( \mathcal D_p (\tau /2) \cup \mathcal D_q (\tau /2) \right) \cap \mathcal B _p (\tau /4)$ and adding $\mathcal A$. This is congruent to $\mathcal F \, {\rm mod}(2)$ and of less volume than $\left( \mathcal D_p (\tau /2) \cup \mathcal D_q (\tau /2) \right) \cap \mathcal B _p (\tau /4)$ if $p_n$ is sufficiently close to $p$. This is impossible since $\mathcal F$ is area-minimizing ${\rm mod}(2)$. This proves Claim A.
\end{proof}

Then, Claim A above proves the theorem.
\end{proof}

\section{Maximum Principle at Infinity}\label{Sect:MPI}

We begin this section by extending Mazet's Theorem to dimensions up to seven.

\begin{theorem}\label{MPI}
Let $\amb^n$, $3 \leq n \leq 7$, be a complete manifold of bounded geometry. Suppose $\Sigma \subset \amb^n$ is a two-sided properly embedded minimal hypersurface that has an $\epsilon$-tube ${\rm Tub}_\epsilon (\Sigma)$ that is embedded, well-oriented, and of bounded extrinsic geometry. If $\Sigma$ is parabolic, then the Maximum Principle at Infinity holds for $\Sigma$. That is, if $S$ is a proper minimal hypersurface such that $\Sigma \cap S = \emptyset$ and $S \cap {\rm Tub}_\epsilon (\Sigma) \neq \emptyset$, then $S$ is an equidistant hypersurface to $\Sigma$.
\end{theorem}

We outline Mazet’s proof and highlight the necessary changes to generalize it to dimensions up to seven. Let $S$ be the properly immersed minimal hypersurface entering the $\epsilon$-tube of the parabolic minimal hypersurface $\Sigma$.

Mazet begins by explaining that if $S$ is stable, his proof will proceed in a manner he will consider later. He first discusses his proof when $S$ is not stable. In this case, $S$ yields a stable $S'$ to work with instead of the original $S$. Then he continues the proof with this new $S'$. The important properties used of $S'$ are the area and curvature estimates for stable minimal surfaces in dimension three: they only depend on the bounded ambient geometry and the distance to the boundary, not on local area bounds. Mazet then proceeds with his proof using this $S'$ (cf. \cite[Section 5.2]{LMaz13}).

In higher dimensions, stability is not enough to obtain such area and curvature estimates, and we must replace the stable $S'$ with an area-minimizer ${\rm mod}(2)$ in the homology class, as described in Lemma \ref{LemF}, which possesses area and curvature estimates as explained in Section \ref{Sect:Mini}.

We proceed with the construction of such $S'$. First, observe that, if $S \cap {\rm Tub}_\epsilon (\Sigma) \neq \emptyset$ contains a compact connected component $S_1 \subset S \cap {\rm Tub}_\epsilon (\Sigma)$, then $S$ is an equidistant hypersurface of $\Sigma$. This follows from the Maximum Principle and the hypothesis that ${\rm Tub}_\epsilon (\Sigma)$ is well-oriented (cf. \cite[Section 6.2.1]{LMaz13}). Assume that $S$ is not an equidistant hypersurface of $\Sigma$. Then one has two possibilities:
\begin{enumerate}
\item[(1)] {\bf $S_1$ is non-compact and not stable:} The construction of $S'$ follows the lines of \cite[Section 5.1.2 and Section 5.1.3]{LMaz13} in the minimal case by replacing the construction of compact stable minimal surfaces with compact area-minimizing ${\rm mod}(2)$ hypersurfaces using Theorem \ref{ThCurrent} and taking limits now using Lemma \ref{LemF}.

\item[(2)] {\bf $S_1$ is non-compact and stable:} For higher dimensions, we must explain what to do if the minimal hypersurface $S_1$ one start with was stable; $S_1$ may not have local area bounds. Here, we use a trick discovered by A. Song \cite[Lemma 17]{ASon18}. Let $p \in S_1$ be a point where a neighborhood of $S_1$ is embedded. There is a small extrinsic metric ball $\mathcal B_p (\delta)$ of $\Sigma \times [0, \epsilon]$, centered at $p$ with radius $\delta >0$ small, and a smooth Riemannian metric on $\Sigma \times [0, \epsilon]$ equal to the original metric on the complement of $\mathcal B$, and there is an unstable minimal hypersurface $\tilde S$ with $\partial \tilde S = \partial S_1$ in the compact cylinder $K_n \times [0, \epsilon]$; $K_n$ is a compact subset of $\Sigma$. Finally, proceed as above with this unstable $\tilde S$.
\end{enumerate}

Hence, in any case, we can construct a non-compact area-minimizing ${\rm mod}(2)$ hypersurface $S'$ possibly contained in a smaller tube of $\Sigma$, that is, $S' \subset {\rm Tub}_{\epsilon'} (\Sigma)$ for $\epsilon' \in (0, \epsilon)$ and $\partial S' \subset \Sigma (\epsilon')$. Then, area and curvature estimates shown in Section \ref{Sect:Mini} prove that $S'$ is (locally) quasi-isometric to $\Sigma$, possibly replacing $\epsilon'$ with a smaller $\tilde \epsilon \in (0, \epsilon')$, following \cite[Section 6.2.1]{LMaz13}. Hence, since $\Sigma$ is parabolic, $S'$ is parabolic with boundary \cite[Proposition 1]{LMaz13}. One then constructs a superharmonic function $h$ on $S$ as in \cite[Section 6.1 and page 824]{LMaz13} of the form $f \circ d$, where $d := {\rm dist}_{\amb}(\cdot, \Sigma)$. Then, $h$ must be constant, a contradiction. So $S$ is an equidistant hypersurface to $\Sigma$. This finishes the proof of Theorem \ref{MPI}.

\subsection{Existence of an embedded $\epsilon$-tube}

We proceed in this section by examining the conditions under which a properly embedded minimal hypersurface has an embedded $\epsilon$-tube. Let $\Sigma^{n-1} \subset \amb^n$ be a two-sided properly embedded $C^2-$hypersurface and $\amb $ complete. Define the map $\Psi : \Sigma \times \r \to \amb$ given by 
\begin{equation}\label{F}
\Psi (p,t) := {\rm exp}_p (t N(p))
\end{equation}
and, since $\Sigma$ is at least $C^2$, there exists an open neighborhood $\Omega$ of $\Sigma$ in $\amb$ formed by non-intersecting minimal geodesics starting normally from $\Sigma$. Hence, by possibly choosing a smaller neighborhood still denoted by $\Omega$, the map $\Psi^{-1}$ is well-defined and smooth in $\Omega $, and there exists an open subset $\Lambda \subset \Sigma \times \r $ with the property that if $(p,t) \in \Lambda$, then also $(p,s) \in \Lambda$ for all $ |s| < |t|$. Hence, the map $\Psi_{|\Lambda} : \Lambda \to \Omega $ is a diffeomorphism and, for every point $p \in \Omega$, there exists a unique point of minimum extrinsic distance to $\Sigma$; that is, $\Omega$ has the {\it unique projection property} (see \cite{LAmbCMan98, LAmbHSon96}). Let $\set{K_i}_{i \in \mathbb N}$ be a compact exhaustion of $\Sigma$ so that $K_i \subset K_{i+1}$ and define 
\begin{equation}\label{Epsiloni}
\epsilon_i := {\rm inf}_{K_i} {\rm sup} \set{t > 0 \, : \, \, (p, s) \in \Lambda \text{ for all } 0 < |s| < t \text{ and } p \in K_i}.
\end{equation}

Observe that if $\epsilon := {\rm inf}\set{\epsilon_i  \, : \, \, i \in \mathbb{N}} > 0$, the above construction gives an embedded $\epsilon$-tube. The existence of an embedded $\epsilon$-tube is crucial when we want to apply Theorem \ref{MPI}. In this sense, following ideas from \cite{WMeeHRos08, HRos01}, we can assert: 

\begin{theorem}[Tubular Neighborhood Theorem]\label{ThETube}
Let $(\amb^n, g)$, $3 \leq n \leq 7$, be a complete manifold of bounded geometry and $\Ricb \geq 0$. Let $\Sigma \subset \amb$ be a properly embedded, two-sided, parabolic, minimal hypersurface of bounded second fundamental form. Then, there exists $\epsilon > 0$ (depending only on the bounds for the sectional curvatures and injectivity radius of $\amb$ and the second fundamental form of $\Sigma$) such that ${\rm Tub}_\epsilon (\Sigma)$ is embedded, well-oriented and of bounded extrinsic geometry (cf. Definition \ref{ETube}).
\end{theorem}

\begin{remark}
An ideal triangle in the hyperbolic plane shows the hypothesis of non-negative Ricci curvature is necessary.
\end{remark}

\begin{proof}[Proof of Theorem \ref{ThETube}]
Since there exists a constant $C > 0$ such that ${\rm inj}(\amb) \geq 1/C$, $\abs{\Kb_{ij}} \leq C$, and $\abs{A} \leq C$, there exists $\delta’ > 0$ (depending only on $C$) such that the map $\Psi: \Sigma \times (-\delta’, \delta’) \to \amb$, given by \eqref{F}, is smooth and for each $p \in \Sigma$, $\Psi: D_p (\delta’) \times (-\delta’, \delta’) \to \amb$ is a diffeormophism onto its image. Let $\Psi^*g$ denote the pullback metric on $\Sigma \times (-\delta, \delta)$ induced by the map $\Psi: \Sigma \times (-\delta’, \delta’) \to (\amb, g)$.

Observe that $(\Sigma \times [-\delta, \delta], \Psi^*g)$, for any $\delta \in (0, \delta’)$, defines an embedded $\delta$-tube for $\Sigma(0) := \Sigma \times \set{0}$. Moreover, by the First Variation Formula for the Mean Curvature \eqref{1VH} and the fact that the Ricci curvature of $\amb$ is non-negative, the mean curvature of $\Sigma(t) = \Sigma \times \set{t}$ satisfies $H’(t) \geq 0$ for all $t \in (-\delta, \delta)$ (recall Lemma \ref{Foliation}); this implies that $(\Sigma \times [-\delta, \delta], \Psi^*g)$ is well-oriented. Also, by \cite[pages 179-180]{HKar87} and the fact that the geometry of $\amb$ and $\Sigma$ are bounded by $C$ from above, all the equidistants $\Sigma(t)$ have a bounded second fundamental form and are quasi-isometric to $\Sigma$ (depending on $C$ and $\delta$). That is, $(\Sigma \times [-\delta, \delta], \Psi^*g)$ has bounded extrinsic geometry (cf. Definition \ref{ETube}). Now, $\Sigma(0) \subset \Psi^{-1}(\Sigma)$ and $\Psi_{|\Sigma(0)}$ is injective.

\begin{quote}
{\bf Claim A:} {\it If there are components of $\Psi^{-1}(\Sigma)$ other than $\Sigma (0)$, then $\amb$ is a mapping torus over $\Sigma$. That is, there exists $\tau > 0$ such that $\amb$ is isometric to $\Sigma \times [0, \tau]$, with the product metric, where $\Sigma \times \set{0}$ is identified with $\Sigma \times \set{\tau}$ by an isometry. In this case, ${\rm Tub}_\epsilon (\Sigma)$ is embedded, well-oriented, and of bounded extrinsic geometry (cf. Definition \ref{ETube}) for any $\epsilon < \tau / 2$.}
\end{quote}
\begin{proof}[Proof of Claim A]
Let $E \subset \Psi^{-1}(\Sigma) \setminus \Sigma (0)$ be such a component. $E$ is a proper minimal hypersurface in $(\Sigma \times [-\delta, \delta], \Psi^*g)$ disjoint from $\Sigma (0)$. Since $\Sigma$ is proper, for some $p \in \Sigma$, we can suppose $E$ is the nearest component of $\Psi^{-1}\left(\Sigma\right) \setminus \Sigma (0)$ in the extrinsic distance in $(\Sigma \times [-\delta, \delta], \Psi^*g)$ at one side of $\Sigma (0)$ (cf. Figure \ref{Figure2}).

By Theorem \ref{MPI}, such a component $E \subset \Psi^{-1}(\Sigma)$ is an equidistant of $\Sigma(0)$; so $E = \Sigma(\tau)$ for some $\tau \in (0, \delta)$. Thus, by Lemma \ref{Foliation}, the metric on $\Sigma \times [0, \tau]$ is a product metric $\Psi^*g = \sigma + dt^2$, $\sigma$ the induced metric on $\Sigma (0)$.

\begin{figure}[htbp]
\begin{center}
\includegraphics[scale=0.2]{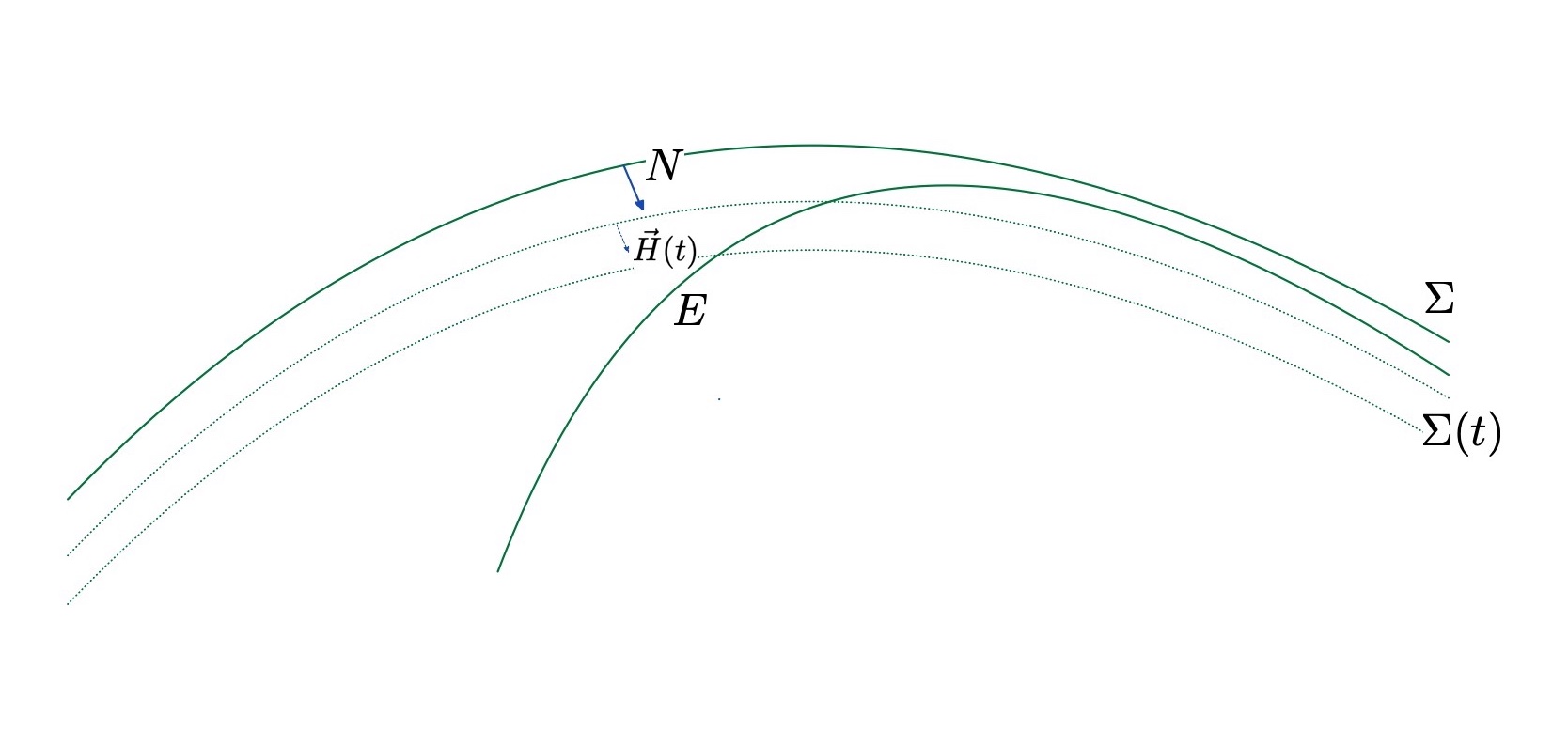}
\caption{$E$ is a proper minimal hypersurface in $(\Sigma \times [-\delta, \delta], \Psi^*g)$}
\label{Figure2}
\end{center}
\end{figure}

Moreover, $\Psi(\Sigma(0)) = \Sigma = \Psi(\Sigma(\tau))$, and hence, $\Psi(\Sigma \times [0, \tau])$ is open and closed in $\amb$, which implies that $\amb$ is a mapping torus over $\Sigma$ and proves Claim A.
\end{proof}

So suppose from now on that $\Psi^{-1}(\Sigma) = \Sigma (0)$. Let $0 < \tau < \delta$ and assume $\Psi: \Sigma \times (-\tau, \tau) \to \amb$ is not an embedding. We will show this gives a contradiction if $\tau$ is sufficiently small.

Let $z_1, z_2 \in \Sigma \times (-\tau, \tau)$, $\Psi(z_1) = \Psi(z_2)$, $z_1 \neq z_2$. We have $z_k = (x_k, t_k)$, $k = 1, 2$, with $|t_k| < \tau$. Since $\Psi: D_p (\delta) \times (-\tau, \tau) \to \amb$ is injective for all $p \in \Sigma$, we have $x_i \not\in D_{x_j}(\delta)$ for $i \neq j$. Also,
\begin{equation}\label{DistF}
\begin{split}
{\rm dist}_{\amb} (\Psi(x_1, 0), \Psi(x_2, 0)) & \leq {\rm dist}_{\amb}(\Psi(x_1, 0), \Psi(z_1)) + {\rm dist}_{\amb}(\Psi(z_1), \Psi(x_2, 0)) \\[2mm]
& = {\rm dist}_{\amb}(\Psi(x_1, 0), \Psi(z_1)) + {\rm dist}_{\amb}(\Psi(z_2), \Psi(x_2, 0)) \\[2mm]
&= |t_1| + |t_2| < 2 \tau.
\end{split}
\end{equation}

Now, $\bar \Psi := \Psi_{| \mathcal{B}_{x_1}(\delta) \times (-\delta, \delta)}$ is an isometry onto an open neighborhood $\mathcal U \subset \amb$ of $x_1$. Choose a large integer $k \in \mathbb{N}$ so that for $\tau = \delta / k$, the geodesic ball of $\amb$ centered at the point $x_1 = \Psi(x_1, 0) \in \amb$ of radius $2 \tau > 0$, $\mathcal B_{x_1} (2 \tau)$, is contained in $\mathcal U$. So, $x_2 = \Psi(x_2, 0) \in \mathcal B_{x_1} (2 \tau) \cap \Sigma$ by \eqref{DistF}. The point $y := \bar \Psi^{-1}(\Psi(x_2, 0))$ is in $D_{x_1}(\delta) \times (-\delta, \delta)$ and on a connected component of $\Psi^{-1}(\Sigma)$ disjoint from $D_{x_1}(\delta) \times (-\delta, \delta)$. This is a contradiction.

Therefore, there exists an $\epsilon > 0$, only depending on $C$, such that ${\rm Tub}_\epsilon (\Sigma)$ is embedded, well-oriented, and of bounded extrinsic geometry. This finishes the proof.
\end{proof}

\section{Splitting and non-separating hypersurfaces}\label{Sect:Spli}

We deal first with quasi-embedded non-separating hypersurfaces $\Sigma \subset \amb$, loosely speaking, those that do not separate $\amb $ into two components. When $\Sigma$ is embedded and compact, we can cut-open the manifold $\amb$ through $\Sigma$ using a local diffeomorphism to obtain a new manifold $\anb$ with two boundary components, each one diffeomorphic to $\Sigma$, and ${\rm int}(\anb)$ isometric to $\amb \setminus \Sigma$ (see \cite{LMazHRos17} for details). We will extend this for properly embedded hypersurfaces, not necessarily compact.

\begin{proposition}\label{NonSeparating}
Let $\Sigma ^{n-1}\subset \amb ^n$ be a connected, non-separating, two-sided properly embedded hypersurface in a complete Riemannian manifold $\amb$. There is a Riemannian manifold $\anb $ with two boundary components $\Sigma _1$ and $\Sigma_2$ and a local diffeomorphism $\varphi : \anb \to \amb $ such that $\varphi : \anb \setminus (\Sigma _1 \cup \Sigma _2) \to \amb \setminus \Sigma$ is a diffeomorphism and $\varphi : \Sigma _i \to \Sigma$ is a diffeomorphism for $i=1,2$.
\end{proposition}
\begin{proof}
Let $\Omega$ be open neighborhood of $\Sigma$ in $\amb$ formed by non-intersecting minimal geodesics starting normally from $\Sigma$. Hence, by possibly choosing a smaller neighborhood still denoted by $\Omega$, the map $\Psi^{-1}$, given by \eqref{F}, is well-defined and smooth in $\Omega $, and there exists an open subset $\Lambda \subset \Sigma \times \r $ with the property that if $(p,t) \in \Lambda$, then also $(p,s) \in \Lambda$ for all $ |s| < |t|$. Take a compact increasing exhaustion $\set{K_i} _{i \in \mathbb N}$ of $ \Sigma$ and define $\epsilon _i >0$ given by \eqref{Epsiloni}. 
Now, we can construct a positive smooth function $f : \Sigma \to \r^+$ so that $2 f(p) < \epsilon _1$ for all $p \in K_1$ and $2 f(p) < \epsilon _i$ for all $p \in K_{i} \setminus K_{i-1}$ and $i \geq 2$. Then, define the open sets using \eqref{F} as
$$ \mathcal U ^\pm := \set{ \Psi(p,t) \in \amb : \, (p, t) \in  \Lambda   \text{ so that }  0< \pm t  < f(p)} $$then $\mathcal U ^\pm \subset \Omega \setminus \Sigma$ and, up to shrinking $\Omega$ if necessary, we can assume that $\mathcal U ^- \cap \mathcal U ^+ = \emptyset$. Set $\amb _f := \amb \setminus \left( \overline{\mathcal U^- \cup \mathcal U^+}\right)$ and consider the open sets 
$$\Lambda ^\pm := \set{ (p,  t ) \in \Sigma \times \r ^\pm \, : \, \, f(p) < \pm t < 2 f(p) }$$and
$$\tilde \Lambda ^\pm := \set{ (p, t ) \in \Sigma \times \r \, : \, \, 0 \leq \pm t < f(p) }$$

\begin{figure}[htbp]
\begin{center}
\includegraphics[scale=0.2]{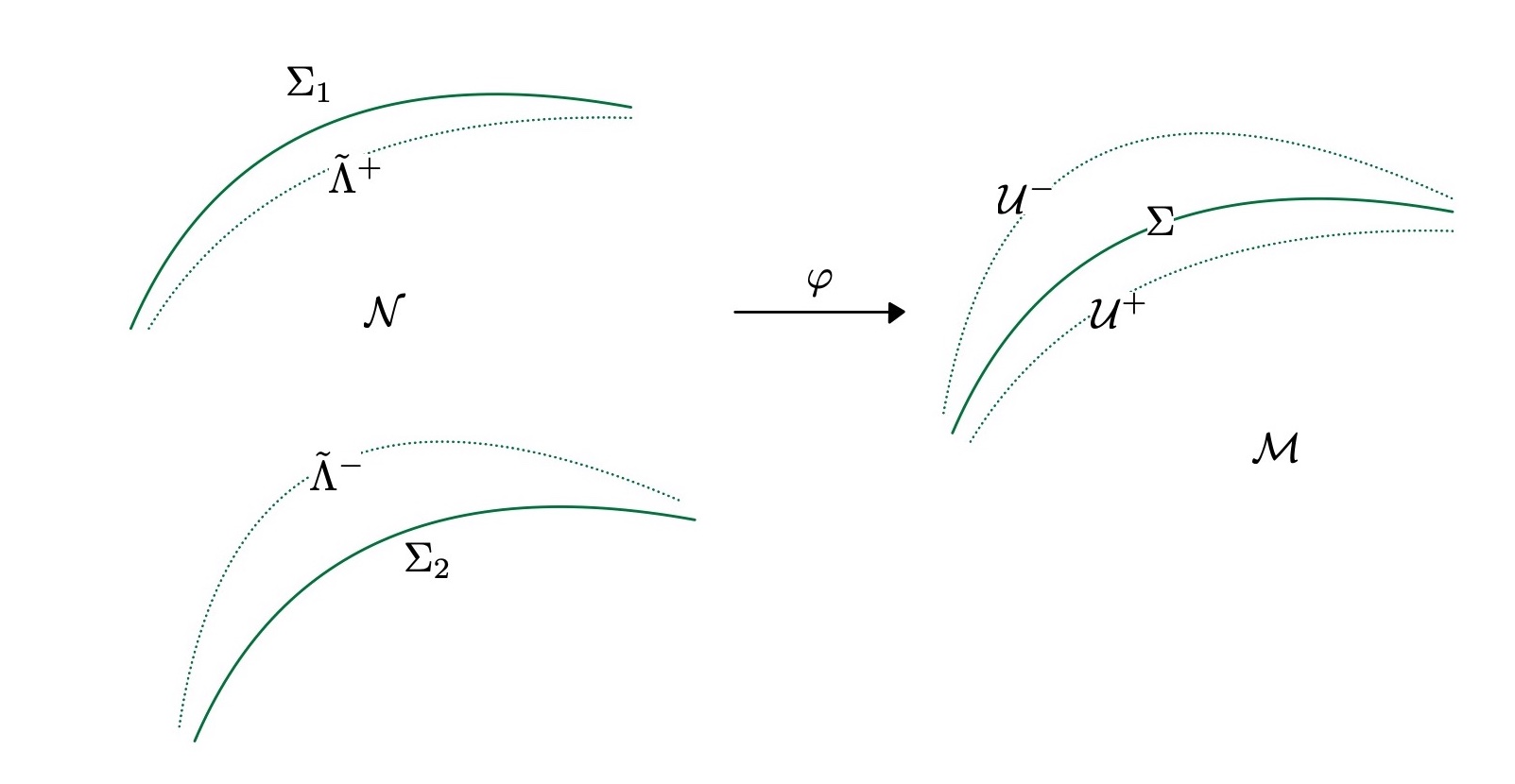}
\caption{$\varphi : \anb \setminus (\Sigma _1 \cup \Sigma _2) \to \amb \setminus \Sigma$ is a diffeomorphism}
\label{Figure3}
\end{center}
\end{figure}

Define $\anb$ as the quotient, endowed with the quotient metric, of the disjoint union of $\amb _f$, $\tilde \Lambda ^+$ and $\tilde \Lambda ^-$ by the identification $(p,t) \equiv \Psi(p,t) \in \amb _f$ for $(p,t) \in \Lambda ^\pm $. The map $\varphi$ is the identity on $\amb _f$ and $\Psi$ on $\tilde \Lambda^+ $ and $\tilde \Lambda ^- $. $\Sigma _1 $ and $\Sigma _2$ are the two copies of $\Sigma \times \set{0}$ (cf. Figure \ref{Figure3}).
\end{proof}

Proposition \ref{NonSeparating} is a way to construct a new manifold with mean convex boundary. This is a good domain to construct (homologically) area-minimizing hypersurfaces. Thanks to the structure of stable hypersurfaces in four dimensional manifolds of non-negative sectional curvature and scalar curvature bounded below by a positive constant that implies parabolicity (cf. \cite{OChoCLiDStr22} and \cite[Theorem 1.2]{OMunJWan22}), we can prove:

\begin{theorem}\label{ThNonSeparating}
Let $(\amb^n  ,g)$, $n=4$, be a complete orientable manifold of bounded geometry, non-negative sectional curvatures, $\Kb _{ij} \geq 0$, and scalar curvature bounded away from zero, $\Sb \geq c >0$. Let $\Sigma \subset \amb$ be a properly quasi-embedded, two-sided, minimal hypersurface. Assume $\Sigma$ does not separate, then:
\begin{itemize}
\item[{\rm (i)}] $\Sigma$ is embedded.
\item[{\rm (ii)}] $\amb ^4$ is a mapping torus; i.e., there exists ${\bf d} >0 $ such that $\amb ^4$ is isometric to $\Sigma \times [0, {\bf d}]$, with the product metric, where $\Sigma \times \set{0}$ is identified with $\Sigma \times \set{{\bf d}}$ by an isometry.
\item[{\rm (iii)}] If $\tilde \Sigma $ is an orientable, properly immersed minimal hypersurface in $\amb \setminus \Sigma$, then $\tilde \Sigma = \Sigma \times \set{ t} $ for some $t \in [0, {\bf d}]$.  
\end{itemize}
\end{theorem}
\begin{proof}
We will first prove the result assuming $\Sigma$ is embedded and stable. 

\begin{quote}
{\bf Case 1:} {\it If $\Sigma $ is embedded and stable, then item ${\rm (ii)}$ and ${\rm (iii)}$ hold.}
\end{quote}
\begin{proof}[Proof of Case 1]
In this case, $\Sigma $ is totally geodesic and parabolic. Hence, Theorem \ref{ThETube} implies that there exists $\epsilon >0$ so that ${\rm Tub}_\epsilon (\Sigma)$ is embedded, well-oriented and of bounded extrinsic geometry. 

Since $\Sigma$ is non-separating, by Proposition \ref{NonSeparating}, there is a Riemannian manifold $ \anb $ with two boundary components $\Sigma _1$ and $\Sigma_2$ and a locally invertible smooth map $\varphi : \anb \to \amb $ such that $\varphi : \anb \setminus (\Sigma _1 \cup \Sigma _2) \to \amb \setminus \Sigma$ is a diffeomorphism and $\varphi : \Sigma _k \to \Sigma$ is a diffeomorphism for $k=1,2$. Since $\Sigma$ has an embedded $\epsilon -$tube, in the construction of $\varphi $, given in Proposition \ref{NonSeparating}, we can take $f \equiv \epsilon $. Consider the open sets 
$$ \Omega _k  : = \set{ p \in  \anb \, : \quad d_{\anb} (p , \Sigma _k) < \epsilon  } \text{ for } k=1,2.$$

Fix $q _1 \in \Sigma _1$, ${\bf d}:= {\rm dist}_{\anb}( q _1,\Sigma _2)$ and let $ \gamma :[0, {\bf d}] \to  \anb $ be a minimizing geodesic such that $|\gamma| = {\bf d} = {\rm dist}_{ \anb}( q _1, q _2)$, $ \gamma (0):= q _2 \in \Sigma _2$, $ \gamma$ orthogonal to $\Sigma _2$ at $ q _2$ and $\gamma([0, {\bf d}]) \cap \left( \Sigma _1 \cup \Sigma _2 \right) = \set{q_1 , q_2}$. 

\begin{quote}
{\bf Claim A:} {\it If $\Omega _1 \cap \Omega _2 \neq \emptyset $, then $\amb$ is a mapping torus.}
\end{quote}
\begin{proof}[Proof of Claim A]
Assume $\Omega _1 \cap \Omega _2 \neq \emptyset $. Set $\mathcal S _k = \partial \Omega _k \setminus \Sigma _k$, $k=1,2$, and assume $\mathcal S _1\cap \Omega _2\neq \emptyset$.

If some connected component of $\mathcal S _1\cap \Omega _2$ is compact, the Maximum Principle applied to this component at the closest point to $\Sigma _2$ shows that $\mathcal S _1$ is an equidistant of $\Sigma _2$. Since the orientation of $\mathcal S _1 \cap \Omega _2$ that makes the mean curvature non-negative points towards $\Sigma _2$ and the orientation of the equidistant of $\Sigma _2$ that makes the mean curvature non-negative points away from $\Sigma _2$, by \eqref{1VH}, the only possibility is that $\mathcal S _1\cap \Omega _2$ is minimal. Hence, $\widetilde{\mathcal S} := \mathcal S _1$ is minimal and equidistant of $\Sigma _1$ and $\Sigma _2$.

% and, by First Variation Formula for the Mean Curvature as in Theorem \ref{ThETube}, each equidistant to $\Sigma _j$ at distance less than $\epsilon /2$, $\Sigma _j (t)$, is totally geodesic and $\Ricb (N_t,N_t) =0 $ along $\Sigma _j (t)$. 

If all connected components of $\mathcal S _1 \cap \Omega _2$ are non-compact, reasoning as in Theorem \ref{MPI}, one shows there is a area-minimizing ${\rm mod}(2)$ hypersurface $\mathcal S '$ in $\Omega _2 \setminus \Sigma _2$ such that $\partial \mathcal S ' \subset \Sigma _2 (\epsilon ')$, $\epsilon ' < \epsilon /2$. This gives that the original $\widetilde{\mathcal S} := \mathcal S _1 \cap \Omega _2$ is an equidistant of $\Sigma _2$. 

Thus, in any case, there exists a minimal hypersurface $\widetilde{\mathcal S} \subset \Omega _1 \cap \Omega _2$ that is equidistant of $\Sigma _1$ and $\Sigma _2$. Hence, any point $p \in \anb$ belongs either to $\widetilde{\mathcal S}$ or one of the the regions between $\widetilde{\mathcal S} $ and $\Sigma _k$ in $\Omega _k$ ($k=1,2$); that is, $\anb = \Omega _1 \cup \Omega _2$. Lemma \ref{Foliation} then implies that $\anb$ is isometric to the Riemannian product $\Sigma _1 \times [0, {\bf d}]$. Claim A is proved.
\end{proof}

Assume now that $\Omega _1 \cap \Omega _2 = \emptyset$. Consider the manifold with boundary $ \widetilde \anb = \anb \setminus (\Omega _1 \cup \Omega _2)$, whose boundary components are denoted by $\mathcal S _k = \partial \Omega _k \setminus \Sigma _k$, $k=1,2$. Let $\tilde \gamma \subset \tilde \anb$ be a geodesic segment of $\gamma$ joining a point $p_1 \in \mathcal S_1$ to a point $p_2 \in \mathcal S _2$; moreover, we know $|\tilde \gamma| \leq  {\bf d} - \epsilon $. By Lemma \ref{LemF}, there is a two-sided properly embedded area-minimizing ${\rm mod}(2)$ hypersurface $\mathcal F _1\subset \tilde  \anb$, homologous to $\mathcal S _1$ and, by  \cite{OChoCLiDStr22} and \cite[Theorem 1.2]{OMunJWan22}, $\mathcal F _1$ is parabolic. Denote by $\anb _1 $ the region of $ \anb$ between $\Sigma _1 $ and $\mathcal F _1$.

If $\mathcal F _1$ touches $\mathcal S _1$ at a point $z$ at one side, then $\mathcal F _1=\mathcal S _1$ by the Maximum Principle; the mean curvature vector of $\mathcal S_1$ at $z $ points into $\tilde \anb$. By the First Variation Formula for the Mean Curvature \eqref{1VH}, the region $\mathcal N _1 \subset \anb$ is foliated by totally geodesic hypersurfaces and Lemma \ref{Foliation} then says the region $\overline{\mathcal N _1}$ is isometric to the product manifold $\Sigma _1 \times [0, \epsilon ]$. 

Recall that $\mathcal S _1 = \Sigma _1(\epsilon)$ is the equidistant of $\Sigma _1$ at distance $\epsilon$. Theorem \ref{ThETube} implies that $\mathcal F _1$ has an embedded $\epsilon-$tube, observe that ${\rm Tub}_\epsilon (\mathcal F_1) \subset \anb$. Let $\mathcal F _1(\epsilon)$ denote the equidistant of $\mathcal F _1$ at distance $\epsilon$ that is in $\anb _1$. If $\mathcal F _1(\epsilon)$ intersects $\Sigma _1 (\epsilon)$ at a point $z$ then first observe that $\mathcal F _1 (\epsilon)$ cannot touch $\Sigma _1 (\epsilon)$ at $z$ (locally at one side), since their mean curvature vectors would point in opposite directions at $z$ so $\mathcal F_1 (\epsilon) = \Sigma _1 (\epsilon)$. Since $\Kb_{ij} \geq 0$ there exists an exhaustion $\set{C_i}_{i \in \mathbb N} \subset \amb$ with mean convex boundary by the work of Cheeger-Gromoll \cite{JCheDGro72}. Then, by Lemma \ref{Foliation}, $\anb _1$ is a Riemannian product bounded by the parallel hypersurfaces $\Sigma _1$ and $\mathcal F_1$; which is what we want to prove. 

So we can suppose $\mathcal F _1(\epsilon)$ traverses $\Sigma _1 (\epsilon)$ locally near $z$. But then an equidistant of $\mathcal F _1$ has a component in ${\rm Tub}_{\epsilon}(\Sigma _1)$. As we saw previously, this implies the component is also an equidistant of $\Sigma _1$, hence, again, $\anb _1$ splits as a Riemannian product bounded by the parallel hypersurfaces $\Sigma _1$ and $\mathcal F _1$. 

Now suppose ${\rm Tub}_{\epsilon}(\mathcal F _1) \cap \overline{\Omega }_1 = \emptyset$. Let $\tilde \anb_2$ be the closure of $\anb _1 \setminus \left( {\rm Tub}_{\epsilon} (\mathcal F_1) \cup {\rm Tub}_{\epsilon}(\Sigma _1)\right)$, $\tilde \anb _2$ has boundary $\mathcal F _1 (\epsilon) \cup \Sigma _1 (\epsilon)$. Let $\tilde \gamma$ be the geodesic segment of $\gamma$ in the interior of $\tilde \anb_2$, joining a point $z_1 \in \mathcal F _1 (\epsilon) $ to a point $x \in \Sigma _1 (\epsilon)$. 

Now we can construct a two-sided properly embedded stable minimal hypersurface $\mathcal F _2$ in $\tilde \anb_2$, $\mathcal F _2$ homologous to $\Sigma _1 (\epsilon)$, and $\mathcal F_2$ has non zero intersection number with $\tilde \gamma$. $\mathcal F _2$ is constructed in the same manner we constructed $\mathcal F _1$. The highest point of $\mathcal F _2 \cap \tilde \gamma$, denoted by $z_2$, is at least $\epsilon$ higher than $z_1$, the highest point of $\mathcal F_1 \cap \gamma$. Also, ${\rm dist}_{\anb} (\mathcal F _2 , \Sigma _1) \geq \epsilon$. 

If ${\rm Tub}_{\epsilon}(\mathcal F _2) \cap \Omega _1 \neq \emptyset$, we conclude as before that the domain bounded by $\Sigma _1$ and $\mathcal F _2$ splits. If the  intersection is empty, we construct a higher least area $\mathcal F _3$ in $\tilde \anb_3$, the closure of $\tilde \anb _2 \setminus \left( {\rm Tub}_{\epsilon} (\mathcal F_2) \cup {\rm Tub}_{\epsilon}(\Sigma _1)\right)$, and continue. After a finite number of steps, we must have ${\rm Tub}_{\epsilon}(\mathcal F _k) \cap \Omega _1 \neq \emptyset$. So, we can suppose $k=1$ and $\anb _1$ is a product manifold over $\Sigma_1$.

Then, consider the manifold $\anb _2 = \overline{\anb \setminus \anb _1}$ bounded by $\mathcal F _1 $ and $\Sigma _2$. They are both stable and parabolic so we can begin the proof again with $\Sigma _1$ replaced by $\mathcal F _1$ and $\Sigma _2$. The region $\anb _2$ differs from the region $\anb $ in the length of $\gamma$, which has been reduced by at least $\epsilon$. So the argument we did can be repeated a finite number of times to finally prove item ${\rm (ii)}$ when $\Sigma$ is stable and embedded. Item ${\rm (iii)}$ follows clearly from item ${\rm (ii)}$. This completes the proof of Step 1.
\end{proof}

Now, we consider the general case:

\begin{quote}
{\bf Step 2:} {\it If $\Sigma $ is quasi-embedded, then $\Sigma$ is embedded and stable.}
\end{quote}
\begin{proof}[Proof of Step 2]

We will first prove:

\begin{quote}
{\bf Claim B:} {\it There exists a properly embedded, orientable, stable minimal hypersurface $\Sigma ' \subset \amb$ that is non-separating.}
\end{quote}
\begin{proof}[Proof of Claim B]
Since $\Kb_{ij} \geq 0$, there exists an exhaustion $\set{C_i}_{i \in \mathbb N} \subset \amb$ with mean convex boundary by the work of Cheeger-Gromoll \cite{JCheDGro72}. Let $\mathcal K \subset \amb$ be the compact set given in Definition \ref{QuasiEmbedded} so that $\Sigma \setminus \mathcal K$ is embedded. Since $\Sigma $ is non-separating (cf. Definition \ref{DefNonSeparating}), for any given $p \in \Sigma \setminus \mathcal K$ there exists a simple closed curve $\gamma : \s ^1 \to \amb$ so that $\gamma \cap \Sigma = \set{p}$ transversally. We can assume that $\gamma \cup \mathcal K \subset C_1$. 

Fix $j \in \mathbb N$ and let $i > j$. Let $\Sigma _i  \subset \Sigma \cap C _i$ be a connected component so that the intersection number of $\gamma $ and $ \Sigma _i$ is odd. Set $\Gamma _i = \partial \Sigma _i$ and $\Sigma _i \subset \Sigma _l $ for all $ l > i$. Let $F_i$ be an embedded least-area hypersurface ${\rm mod}(2)$ relative to $(C_i , \Sigma _i , \Gamma _i)$ given by Theorem \ref{ThCurrent}. $F_i$ is in the compact mean convex region $C_i$ and $\gamma \cap F_i \neq \emptyset$. 

For $i>j$, let $F_{ij} $ be a component of $C_{j}\cap F _i$ such that $F_{ij} \cap \gamma \neq \emptyset$; this component exists since the intersection number of $\Sigma_i$ and $ \gamma$ is odd. Arguing as in Lemma \ref{LemF} and letting first $i$ and then $j$ go to infinity, we obtain a two-sided properly embedded area-minimizing ${\rm mod}(2)$ hypersurface $\Sigma '\subset \amb$, homologous to $\Sigma$, and intersects $ \gamma$. This proves Claim B.
\end{proof}

Now, we apply Step 1 to $\Sigma '$ and we obtain that $\Sigma = \Sigma ' \times \set{t}$, for some $t \in [0,{\bf d}]$. This proves Step 2.
\end{proof}

Finally, Step 1 and Step 2 prove Theorem \ref{ThNonSeparating}.
\end{proof}

%\begin{remark}\label{RemCurEst}
%Until recently, curvature estimates for stable minimal hypersurfaces in dimensions greater than three depend not only on distance from the boundary but also on volume growth hypothesis. Thanks to a recent theorem of Chodosh and Li \cite{OChoCLi22}, where they proved that two-sided stable minimal hypersurfaces in $\r ^4$ are hyperplanes, and a blow-up argument, we now have curvature estimates that only depend on the distance to the boundary when the ambient manifold has bounded geometry and dimension four \cite[Lemma 2.4]{OChoCLiDStr22}. It is expected that two-sided stable minimal hypersurfaces in $\r ^n$ are hyperplanes up to dimension $n \leq 7$, which would imply curvature estimates for stable hypersurfaces that only depend on the distance to the boundary in ambient manifolds of bounded geometry and dimension at most seven. If this were true, Theorem \ref{ThNonSeparating} would be valid for the dimensions $5$, $6$ and $7$ as well. 
%\end{remark}

As a consequence of the above result, we have a Frankel type property in these manifolds: 

\begin{theorem}\label{CorNonSeparating}
Let $(\amb^4  ,g)$ be a complete orientable manifold of bounded geometry, non-negative sectional curvatures, $\Kb _{ij} \geq 0$, and scalar curvature bounded away from zero, $\Sb \geq c >0$. Let $\Sigma _1,  \Sigma _2\subset \amb$ be two properly quasi-embedded, two-sided, minimal hypersurfaces. Then, either $\Sigma _1$ intersects $  \Sigma _2$ or they are parallel.
\end{theorem}

We shall explain first the meaning of parallel in a Riemannian manifold. Usually, {\it parallel} means that the distance between the two hypersurfaces is constant. For us, to be parallel is a stronger condition, we also ask that between the hypersurfaces the manifold has a product structure. 

\begin{definition}\label{Parallel}
Let $\amb ^n$ be a complete Riemannian manifold. We say that two disjoint properly embedded hypersurfaces $\Sigma _1, \Sigma _2\subset \amb $ are parallel if there exists a domain $\mathcal U \subset \amb \setminus \left( \Sigma _1\cup \Sigma_2 \right)$ such that $\partial \mathcal U = \Sigma _1\cup \Sigma _2$ and $\overline{\mathcal U}$ is a product manifold over $\Sigma _1 $ (or $\Sigma _2$); i.e, $\overline{\mathcal U} = \Sigma _1\times [0, {\bf d}]$ endowed with the product metric. Here ${\bf d} := {\rm dist}_{\amb}(\Sigma  _1, \Sigma _2) >0$ and $ \Sigma _2= \Sigma _1 \times \set{{\bf d}}$.
\end{definition}

\begin{proof}[Proof of Theorem \ref{CorNonSeparating}]
Assume that $\Sigma _1 \cap \Sigma _2 = \emptyset$ and we will show that they are parallel. We distinguish two cases:

\begin{quote}
{\bf Case 1:} {\it If either $\Sigma _1$ or $ \Sigma _2$ are non-separating, then $\Sigma _1$ and $  \Sigma _2$ are parallel.}
\end{quote}
\begin{proof}[Proof of Case 1]
Assume one of them is non-separating; say $\Sigma _1$. Then, by Theorem \ref{ThNonSeparating}, $\Sigma _1$ is embedded and $\amb ^4$ is a mapping torus over $\Sigma _1$. In particular, $\Sigma _1$ is stable and hence parabolic (cf. \cite{OChoCLiDStr22} and \cite[Theorem 1.2]{OMunJWan22}). In this case, since $\Sigma _1\cap \Sigma _2= \emptyset$, Theorem \ref{MPI} shows that $ \Sigma _2$ is a leaf in the Riemannian product $\Sigma _1 \times [0, {\bf d}]$. This proves Case 1.
\end{proof}

Next, we focus on the case that $\Sigma _1 $ and $\Sigma _2$ are separating.

\begin{quote}
{\bf Case 2:} {\it If $\Sigma _1$ and $ \Sigma _2$ are both separating, then $\Sigma _1$ and $  \Sigma _2$ are parallel.}
\end{quote}
\begin{proof}[Proof of Case 2]
Since $\Sigma _1\cap \Sigma _2 = \emptyset$, there exists a connected component $\mathcal N \subset \amb \setminus \left( \Sigma _1 \cup \Sigma _2\right)$ such that $\partial \mathcal N := \mathcal S _1 \cup \mathcal S_2 \subset \Sigma _1 \cup  \Sigma _2$, where $\mathcal S _1 \subset \Sigma _1$ and $\mathcal S _2 \subset \Sigma_2 $. $\anb $ is piece-wise mean convex, the smooth pieces intersecting at angles less than $\pi$, that is, the boundary $\partial \anb$ is a good barrier. Fix $q _1 \in \mathcal S _1$, ${\bf d}:= {\rm dist}_{\anb}( q _1,\mathcal S _2)$ and $ \gamma :[0, {\bf d}] \to  \anb $ a minimizing geodesic such that $|\gamma| = {\bf d} = {\rm dist}_{ \anb}( q _1, q _2)$. We can assume that $q_i \in \mathcal S_ i$ belongs to the regular part, $i=1,2$ (cf. Figure \ref{FigQuasiEmbedded}). 

\begin{figure}[htbp]
\begin{center}
\includegraphics[scale=0.25]{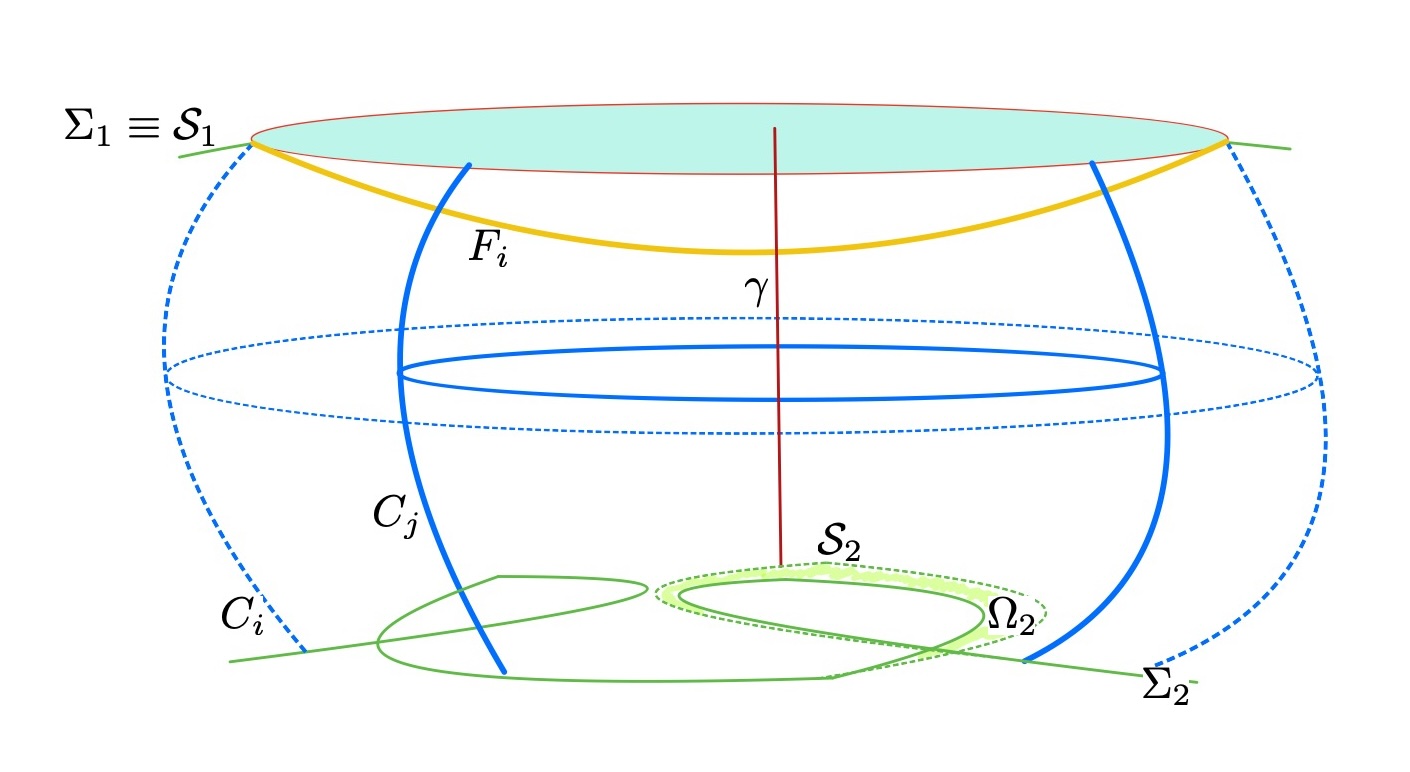}
\caption{Quasi-embedded case}
\label{FigQuasiEmbedded}
\end{center}
\end{figure}

Since $\Kb_{ij} \geq 0$, there exists an exhaustion $\set{C_i}_{i \in \mathbb N} \subset \amb$ with mean convex boundary by the work of Cheeger-Gromoll \cite{JCheDGro72}. By Lemma \ref{LemF}, we obtain a two-sided properly embedded area-minimizing ${\rm mod}(2)$ hypersurface $\mathcal F\subset  \anb$, homologous to $\mathcal S _1$, and intersects $ \gamma$. We distinguish the following possibilities: 
\begin{itemize}
\item {\bf Either $\mathcal F = \mathcal S _1$ or $\mathcal F = \mathcal S _2$:} We assume $\mathcal F = \mathcal S _1$. Then, $\Sigma _1= \mathcal S _1$ is embedded and stable, hence parabolic. Theorem \ref{ThETube} implies that there exists $\epsilon >0$ such that ${\rm Tub}_\epsilon (\Sigma _1)$ is embedded, well-oriented, and of bounded extrinsic geometry. If $\mathcal S_2 \cap {\rm Tub}_\epsilon (\Sigma _1) \neq \emptyset$, Theorem \ref{MPI} show that $\mathcal S _2 = \Sigma _2$ is an equidistant of $\Sigma _2$, in particular, $\Sigma _2$ is embedded. As in Claim A of Theorem \ref{ThNonSeparating}, $\anb$ is a Riemannian product $\Sigma _1 \times [0, {\bf d}]$, where $\Sigma _2 = \Sigma _1 \times \set{{\bf d}}$ and ${\bf d} :={\rm dist}_{\anb} (\Sigma _1 , \Sigma _2)$. If $\mathcal S_2 \cap {\rm Tub}_\epsilon (\Sigma _1) = \emptyset$, we would start the proof again replacing $\Sigma _1$ by $\Sigma _1 (\epsilon)$, the equidistant of $\Sigma_1$ at distance $\epsilon$ in $\anb $, and $\anb$ by $\anb \setminus \overline{{\rm Tub}_\epsilon (\Sigma _1)}$; in this new $\anb$, the distance between the boundary components of $\anb$ is $\epsilon $ smaller. 

\item {\bf $\mathcal F \cap (\mathcal S _1 \cup \mathcal S _2) = \emptyset$ and there exists $\epsilon >0$ such that ${\rm Tub}_\epsilon (\mathcal F)\cap (\mathcal S _1 \cup \mathcal S _2) \neq \emptyset$:} We can assume that $\mathcal S _1 \cap {\rm Tub}_\epsilon (\mathcal F) \neq \emptyset$. Then, as we showed above, either the equidistant $\mathcal F (\epsilon)$ of $\mathcal F$ at distance $\epsilon > 0$ in $\anb$ is tangential to $\mathcal S_1$ at a (regular) point $z$, or they are transverse; in any case and using Theorem \ref{MPI}, an equidistant of $\mathcal F$ is also an equidistant of $\Sigma _1 = \mathcal S _1$. Hence, as in Claim A of Theorem \ref{ThNonSeparating}, the region $\anb _1 \subset \anb$ between $\Sigma _1 $ and $\mathcal F$ is isometric to a product manifold over $\Sigma _1$. Now, we would start the proof again replacing $\anb$ by $\anb \setminus \overline{\anb _1}$; in this new $\anb$, the distance between the boundary components of $\anb$ is $\epsilon $ smaller. 

\item {\bf There exists $\epsilon >0$ such that ${\rm Tub}_\epsilon (\mathcal F)\cap (\mathcal S _1 \cup \mathcal S _2) = \emptyset$:} Denote now by $\mathcal N _1$ the domain in $\anb $ bounded by $\mathcal S _1$ and $\mathcal F$. Arguing as above; in a finite number of steps, say $m$, depending on $\epsilon >0$ given in Theorem \ref{ThETube} and the length $\abs{\gamma _1 }$, where $\gamma _1$ is the geodesic segment of $\gamma$ in the interior of $ \anb_1$, we can construct $\mathcal F_k$, $k \in \set{1, \ldots , m}$, pairwise disjoint, two-sided, properly embedded area-minimizing ${\rm mod}(2)$ hypersurfaces, each $\mathcal F_k$ with an embedded, well-oriented and of bounded extrinsic geometry $\epsilon -$tube, ${\rm Tub}_\epsilon ( \mathcal F_k) $; moreover, $\mathcal S_1 \cap {\rm Tub}_\epsilon ( \mathcal F_{m}) \neq \emptyset$. So, we can suppose $m=1$ and $\anb _1$ is a product manifold over $\Sigma_1$. Now, we would start the proof again replacing $\anb$ by $\anb \setminus \overline{\anb _1}$; in this new $\anb$, the distance between the boundary components of $\anb$ is $\epsilon $ smaller. 
\end{itemize}

Therefore, in a finite number of steps (depending on ${\bf d}$ and $\epsilon >0$ of Theorem \ref{ThETube}), we can show that there exists a properly embedded area-minimizing ${\rm mod}(2)$ hypersurface $\mathcal F \subset \anb $ such that $\anb \setminus \mathcal F = A \cup B$, where $B$ is a product manifold with boundaries $\mathcal S _1 = \Sigma _1 $ and $\mathcal F$; and $A$ is the region between $\mathcal S _2$ and $\mathcal F$ in $\anb $ and such that ${\rm Tub}_\epsilon (\mathcal F) \cap \mathcal S _2 \neq \emptyset$. Therefore, as before, $A$ is also a product manifold over $\mathcal S _2$ (or $\mathcal F$). Hence, $\overline{\anb }$ is isometric to the product manifold $\Sigma_1  \times [0, {\bf d}]$ and $ \Sigma_2  = \Sigma _1 \times \set{ {\bf d} }$. This proves Case 2.
\end{proof}

Finally, Cases 1 and 2 prove Theorem \ref{CorNonSeparating}.
\end{proof}

\section{Parabolic area-minimizing hypersurfaces}\label{Sect:Para}

Let $\amb ^n$, $3 \leq n \leq 7$, be a complete orientable manifold of bounded geometry and $\Sigma^{n-1} \subset \amb^n$ be a complete, two-sided, non-compact hypersurface; let $N$ be a unit normal along $\Sigma$. Assume that $\Ricb \geq 0$ and $\Sigma \subset \amb$ is area-minimizing ${\rm mod}(2)$ and parabolic. Then, $\Sigma$ is properly embedded, has an embedded, well-oriented, and of bounded extrinsic geometry $\bar \epsilon-$tube ${\rm Tub}_{\bar \epsilon} (\Sigma)$ (cf. Theorem \ref{Mod2Proper}), for some $\bar \epsilon >0$. Moreover, $\Sigma$ is totally geodesic and $\Ricb (N,N) \equiv 0$ along $\Sigma$ (cf. Proposition \ref{PropParabolicStable}).

Fix a point $p \in \Sigma$ and let $D_p(\epsilon) \subset \Sigma$ be the geodesic ball in $\Sigma$ centered at $p\in \Sigma $ of radius $\epsilon >0 $. We might choose $\epsilon _0 >0$ small enough so that $D_p (\epsilon) $ is topologically a ball and $\partial D_p (\epsilon)$ is smooth for all $0< \epsilon <  \epsilon _0 \leq \bar \epsilon$. Let $\set{K_i}_{i\in \np}$ be an exhaustion of $\Sigma$ by relatively compact sets so that $\overline{D_p (\epsilon_0)} \subset K_1$, $K_i \subset K_j$, $i< j$, and $\partial K_i$ is smooth for all $i \in \n$.

Consider the embedding $\Psi : \Sigma \times [-\epsilon , \epsilon ] \to \amb ^n$ given by the exponential map $\Psi(p,t) = {\rm exp}_{p} (t N(p))$. Shrinking $\epsilon _0 $ if necessary, we can assume that any normal geodesic starting at any point $q \in \overline{D_p (\epsilon_0)}$ does not develop focal points in $[-\epsilon _0 , \epsilon _0 ]$. For each $t \in [-\epsilon _0 , \epsilon _0]$ and $\epsilon \in (0, \epsilon _0)$, denote by $\tilde C(\epsilon ,t)$ the translation of $\partial D_p (\epsilon)$ by the exponential map at distance $t$; that is, $ \tilde C (\epsilon , t) = F\left( \partial D_p (\epsilon) ,t \right) $.

Moreover, set $\mathcal A_i (\epsilon) = K_i \setminus \overline{D_p (\epsilon)}$ and consider the harmonic function $f_i^{\epsilon,t} : \mathcal A_i (\epsilon) \to \r$ given by 
\begin{equation*}
\left\{ \begin{matrix}
\Delta f_i ^{\epsilon ,t} =0 & \text{ on } & \mathcal A_i (\epsilon), \\[1mm]
f_i ^{\epsilon ,t} = t & \text{ along } & \partial D_p (\epsilon), \\[1mm]
f_i ^{\epsilon ,t} = 0 & \text{ along } & \partial K_i ,\\
\end{matrix}\right.
\end{equation*}and denote by $S_i ^{\epsilon , t}$ the graph of $f_i^{\epsilon ,t}$ under the exponential map; that is, 
$$ S_i ^{\epsilon ,t} = \set{ {\rm exp} _q (f_i ^{\epsilon ,t} (q)N(q)) \, : \, \, q \in \mathcal A_i (\epsilon) }.$$ 

Then, it is clear that $\partial S_i ^{\epsilon , t} = \tilde C (\epsilon ,t) \cup \partial K_i$. Denote by $\tilde D (\epsilon ,t)$ an area-minimizing hypersurface with boundary $\tilde C(\epsilon , t)$. Now, we need to prove: 

\begin{lemma}\label{Douglas}
Fix $\epsilon \in (0, \epsilon _0)$ and $t \in [-\epsilon _0 , \epsilon _0]$. Then, there exists $i_0 \in \n$, depending on $\epsilon$, $t$ and the upper bound for the sectional curvature of $\amb$, so that
$$ \mathcal H ^{n-1} (S_i ^{\epsilon ,t}) < \mathcal H^{n-1} (K_i) + \mathcal H ^{n-1} (\tilde D (\epsilon ,t)) \text{ for all } i \geq i_0 .$$
\end{lemma}
\begin{proof}
Consider the smooth map $\Psi_i : \mathcal A_i (\epsilon) \to \amb$ given by $\Psi_i (q) = {\rm exp} _q (f_i ^{\epsilon ,t} (q) N(q))$. Then, 
$$ \mathcal H ^{n-1} (S_i ^{\epsilon ,t}) = \int _{\mathcal A_i (\epsilon)} \abs{{\rm Jac }\, \Psi_i} d\Sigma ,$$where $d\Sigma$ is the volume element in $\Sigma$. Hence, we need to estimate the Jacobian $\abs{{\rm Jac} \, \Psi_i}$. Let $q \in \mathcal A_i (\epsilon)$ and $\set{e_j}_{j=1, \ldots , n-1} \in T_q \mathcal A_i (\epsilon )$ be an orthonormal basis so that 
$$e_1 = \frac{\nabla f_i ^{\epsilon,t}}{|\nabla f_i ^{\epsilon,t}|} \text{ and } g \left( e_j , \nabla f_i ^{\epsilon,t}\right) =0 \text{ for } j = 2, \ldots , n-1 ;$$we can do this choice for all $q \in \mathcal A_i (\epsilon) \setminus \mathcal C$, where $\mathcal H^{n-1} (\mathcal C) =0$. For $i=1, \ldots, n-1$, consider the geodesic $\alpha _i :(-\tau , \tau) \to \mathcal A_i (\epsilon )$ with initial conditions $\alpha _j (0) = q$ and $\alpha _j ' (0) = e_j$. Consider the variation by normal geodesics $g_j(s,r) := {\rm exp}_{\alpha _j (s)} (r f_i ^{\epsilon ,t} (\alpha _j (s))N(\alpha _j (s)))$, $(s,r) \in (-\tau , \tau ) \times [0,1]$, and set
$$ \tilde J_j (r) := \left. \frac{d}{ds}\right| _{s=0} g_j(s,r) , \quad r \in [0,1], $$the Jacobi field along the geodesic $r \to g_j (0,r)$ with initial condition $\tilde J _ j (0) = e_j $. Consider also the Jacobi field along the geodesic $r \to g^0_j (0,r)$ associated to the variation $g^0_j (s,r) := {\rm exp}_{\alpha _j (s)} (r N(\alpha _j (s)))$, that is, 
$$ J_j (r) := \left. \frac{d}{ds}\right| _{s=0} g^0_j(s,r), \quad r \in [0,1], $$with, also, initial condition $J_j (0) = e_j$. Observe that, since $\Sigma$ is totally geodesic and $\set{e_j}_{j =1, \ldots ,n-1}$ is orthonormal, then 
$$\set{\gamma _q ' (r), J_1 (r) , \ldots , J_{n-1}(r)}$$ is orthogonal at $q_r:=\gamma _q (r) = {\rm exp}_{q} (r N(q))$ for all $r >0$; in particular, it is a tangent basis along $\gamma _q$ (cf. \cite[page 179]{HKar87}). 

Set $\bar r := (f_i^{\epsilon, t}) (q)$ and $\bar q := \gamma _q (\bar r) \in S _i ^{\epsilon , t}$. Then, $\set{\tilde J _j (1) }_{j=1}^{n-1}$ is a basis of $T_{\bar q} S_i ^{\epsilon ,t}$ and it can be orthogonally decomposed as
\begin{equation}\label{eqa}
 \tilde J _j (1) = J_j (\bar r) + e_j (f_i^{\epsilon, t}) \gamma _q ' (\bar r) 
\end{equation}and hence, by \eqref{eqa}, we get
\begin{equation*}
\tilde J_1 (1) \wedge \ldots \wedge \tilde J_{n-1} (1)  =  J_1 (\bar r) \wedge \ldots \wedge J_{n-1} (\bar r) + e_1 (f_i^{\epsilon, t}) \gamma _q ' (\bar r)  \wedge J_2 (\bar r) \wedge \ldots \wedge J_{n-1} (\bar r),
\end{equation*}and now, since $J_1 (\bar r) \wedge J_2 (\bar r) \wedge \ldots \wedge J_{n-1} (\bar r) $ and $\gamma _q ' (\bar r)  \wedge J_2 (\bar r)\wedge \ldots \wedge J_{n-1} (\bar r)$ are orthogonal, we obtain 
\begin{equation}\label{eqb}
\abs{\tilde J_1 (1) \wedge \ldots \wedge \tilde J_{n-1} (1)}^2 = \abs{J_1 (\bar r) \wedge \ldots \wedge J_{n-1} (\bar r) }^2 \left( 1+ \frac{\abs{\nabla f_i ^{\epsilon,t}}^2 }{\abs{J_1(\bar r)}^2}\right).
\end{equation}

Using now that $\Ricb \geq 0$, it is well-known (see \cite{EHeiHKar78}) that 
\begin{equation}\label{eqc}
 \abs{J_1 (\bar r) \wedge \ldots \wedge J_{n-1} (\bar r) } \leq \abs{e_1 \wedge \ldots \wedge e_{n-1}} =1.
\end{equation}

Thus, since $\abs{{\rm Jac} \, \Psi_i} =  \abs{\tilde J_1 (1) \wedge \ldots \wedge \tilde J_{n-1} (1)}$ and using \eqref{eqb} and \eqref{eqc}, we obtain
\begin{equation}\label{Estimate1}
\mathcal H ^{n-1} (S_i^{\epsilon ,t})  \leq \int _{\mathcal A_i (\epsilon)} \sqrt{1+ \frac{\abs{\nabla f_i ^{\epsilon,t}}^2 }{\abs{J_1(\bar r)}^2}} d\Sigma .
\end{equation}

Finally, we need to estimate the right-hand side of \eqref{Estimate1}. First, since $|J_1(0)| =1$ and any normal geodesic starting at any point $q \in \overline{D_p (\epsilon_0)}$ does not develop focal points in $[-\epsilon _0 , \epsilon _0 ]$, then $\abs{J_1} \geq \lambda >0$, for some $\lambda$ depending only on the upper bound of the sectional curvature. Second, using the inequality $\sqrt{1+x} \leq 1+x$ for all $x \geq 0 $ applied to \eqref{Estimate1}, we obtain 
\begin{equation}\label{eqd}
\begin{split}
\mathcal H ^{n-1} (S_i^{\epsilon ,t}) & \leq \mathcal H ^{n-1} (\mathcal A_i(\epsilon )) + \frac{1}{\lambda ^2} \int _{\mathcal A_i (\epsilon)} \abs{\nabla f_i ^{\epsilon,t}}^2 \, d\Sigma  \\
 &= \mathcal H ^{n-1} (K_i) + \frac{1}{\lambda ^2} \int _{\mathcal A_i (\epsilon)} \abs{\nabla f_i ^{\epsilon,t}}^2 \, d\Sigma - \mathcal H ^{n-1} (D_p (\epsilon)) .
\end{split}
\end{equation}

Thus, since $\Sigma $ is parabolic, there exists $i_0 \in \n$ (cf. \cite[Theorem 10.1]{PLi04}) so that 
$$  \int _{\mathcal A_i (\epsilon)} \abs{\nabla f_i ^{\epsilon,t}}^2 \, d\Sigma  <  \lambda ^2 \mathcal H ^{n-1} (D_p (\epsilon)) \text{ for all } i \geq i_0,$$that applied to \eqref{eqd} implies $\mathcal H ^{n-1} (S_i^{\epsilon ,t}) < \mathcal H ^{n-1} (K_i) $ for $i \geq i_0$. The proof is completed.
\end{proof}

The above lemma will prove a local splitting theorem as in \cite{MAndLRod89} and we include here a sketch of the proof.    

\begin{remark}[Douglas Criterion for the Plateau Problem]
When we want find area-minimizing hypersurfaces (currents) in $\amb$ with two (disjoint) boundary components $\Gamma _1 \cup \Gamma _2$ that is connected; we must ensure that there exists a comparison hypersurface $S \subset \amb$, $\partial S = \Gamma _1 \cup \Gamma _2$, of area less than $\mathcal H^{n-1}(\Sigma _1) + \mathcal H ^{n-1} (\Sigma _2)$, where $\Sigma _i$ are ay least-area hypersurfaces bounding $\Gamma _i$, $i=1,2$. In such case, any area-minimizing hypersurface $\mathcal A \subset \amb $ with $\partial \mathcal A = \Gamma _1 \cup \Gamma _2$ must be connected. 
\end{remark}

\begin{theorem}\label{ThAreaMinimizing}
Let $\amb^n$, $3\leq n \leq 7$, be a complete orientable manifold of bounded geometry and $\Ricb \geq 0$. Let $\Sigma \subset \amb$ be complete, parabolic, oriented, area-minimizing ${\rm mod}(2)$ hypersurface. Then, the universal covering $\widetilde \amb$ of $\amb $ is isometric to $ \Sigma \times \r $ with the product metric. 
\end{theorem}
\begin{proof}
As explained in the beginning of this section, $\Sigma$ is properly embedded and has an embedded, well-oriented, and of bounded geometry  $\bar \epsilon-$tube, for some $\bar \epsilon >0$. Consider the hypersurfaces $S_i ^{\epsilon ,t}$ constructed above. From Lemma \ref{Douglas} and using the Douglas Criterion for the Plateau Problem, for each $i \geq i_0$, there exists a compact connected area-minimizing hypersurface $\tilde S^{\epsilon , t}_i$, smooth up to $n \leq 7$ (see \cite{LSim83,NWic14}), so that $\partial \tilde S ^{\epsilon , t}_ i = \tilde C (\epsilon ,t) \cup \partial K_i $. Moreover, by the Maximum Principle $\tilde S^{\epsilon ,t}_ i \subset \Psi (\Sigma \times (0, t))$. 

Since $\tilde S ^{\epsilon ,t} _i$ is area-minimizing we have area and curvature estimates on compact sets as we did in Section 3. Hence, we can follow the arguments in Section 3 to obtain a subsequence of $\set{\tilde S ^{\epsilon ,t} _i}_{i \in \np}$ that converges to a connected area-minimizing hypersurface $\tilde S ^{\epsilon , t}_\infty$ such that $\partial \tilde S ^{\epsilon , t}_\infty = \tilde C (\epsilon ,t)$. Next, letting $\epsilon \to 0$ and keeping $t \in (0, \epsilon _0]$ fixed,  one obtains an area-minimizing hypersurface $\tilde S ^t _\infty$ in $\amb \setminus \set{\gamma _q(t)}$ such that: 
\begin{itemize}
\item it is embedded;
\item may a priori have infinite topological type;
\item proper by construction;
\item smooth outside $\gamma _q(t)$ and across $\gamma _q(t)$ up to $n \leq 7$ (see \cite[Section 3.4]{NWic14}).
\end{itemize}

Then, following \cite[page 464]{MAndLRod89}, we can show that the family $\set{\tilde S^t _\infty} _{t \in [-\epsilon _0 , \epsilon _0]}$, shrinking $\epsilon _0$ if necessary, forms a $C^0$ foliation of a region of $\amb $, where $\tilde S _\infty ^0 = \Sigma$ is a leaf of this foliation. In particular, all the $\tilde S _\infty ^t $ are quasi-isometric to $\Sigma $ by Schoen-Simon-Yau's curvature estimates, shrinking $\epsilon _0 >0$ if necessary. Therefore, since all the $\tilde S _\infty ^t$ are also stable, $\set{\tilde S^t _\infty} _{t \in [-\epsilon _0 , \epsilon _0]}$ is a foliation by proper, totally geodesic hypersurfaces of a region of $\anb \subset \amb$. Hence, Lemma \ref{Foliation} implies that $\anb$ is isometric to the Riemannian product $\Sigma  \times [-\epsilon _0 , \epsilon _0]$. 

As in \cite[page 465]{MAndLRod89}, we can also show that each leaf of the foliation is area-miminizing (and parabolic); thus we can iterate this process indefinitely to finally prove the result.
\end{proof}

\section*{Appendix}

\begin{quote}
{\bf Theorem A:} {\it Let $(\s ^2 ,g)$ be endowed with a metric of positive curvature. Then, every complete embedded geodesic in $(\s ^2 ,g)$ is compact.}
\end{quote}
\begin{proof}[Proof of Theorem A]
Let $\gamma : \r \to (\s ^2 ,g)$ be a complete (parametrized by arc-length) embedded geodesic. If $\gamma$ is not compact then $\overline{\gamma} = \mathcal L$ is a geodesic lamination with a limit leaf $C$ which, in particular, is an embedded complete geodesic. Also, any $q \in C$ is a limit of points $p _n \in \gamma$ that diverge on $\gamma$.

Let $\vec{n}$ be a unit vector field to $C$. The curvature of $(\s ^2 , g)$ is bounded so there exists $\epsilon >0$ and $\delta >0$ such that for all $x = C (t_0)$, then the map $\psi : [t_0- \delta , t_0+ \delta ]\times [-\epsilon , \epsilon] \to (\s ^2 ,g)$ given by 
$$  \psi (t,s) := {\rm exp}_{C(t)}(s \vec{n}(C (t)))$$is a diffeomorphism and $\psi$ extends to an immersion $\Psi : \r \times [-\epsilon , \epsilon] \to (\s ^2 ,g)$.

For each $s \in [-\epsilon , \epsilon ]$, $C _s : \r \to (\s ^2 ,g)$, given by $C_s (t):= \Psi (t,s)$, is the equidistant at distance $s$ of $C$. Its geodesic curvature vector points away from $C$. 

Think locally of $C$ as horizontal and the $C _s$ as well, the vertical segments are the geodesics $\beta _t : [-\epsilon , \epsilon] \to (\s ^2 ,g)$, given by $\beta _t (s) = \Psi(t,s)$. Pull back the metric of $(\s ^2 ,g)$ to $\r \times [-\epsilon , \epsilon]$ by $\Psi$; call it $\tilde g$ (cf. Figure \ref{FigGeod1}).

\begin{figure}[htbp]
\begin{center}
\includegraphics[scale=0.2]{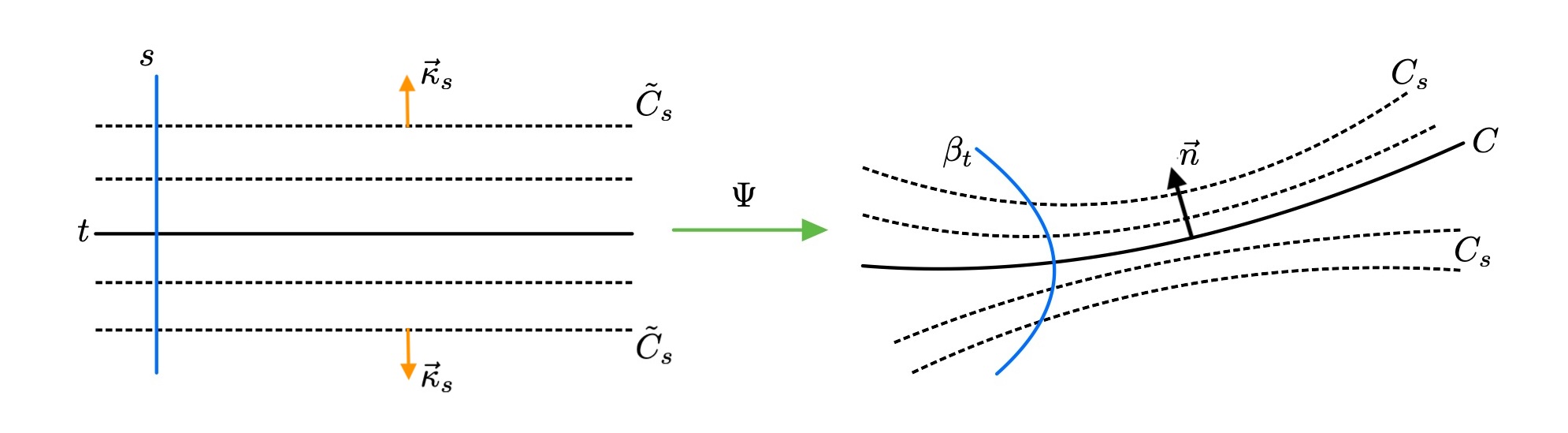}
\caption{The immersion $\Psi$}
\label{FigGeod1}
\end{center}
\end{figure} 

Since $\gamma$ is embedded, the tangent lines to $\gamma$ at $p_n$ converge to the tangent line to $C$ at $q$ as $n \to \infty$. The neighborhood 
$$W(q) := \set{\beta _t (s) \, : \, \, t \in [\bar t - \delta , \bar t + \delta] , \, s \in [0, \epsilon]},$$ lifts isometrically to $ [\bar t - \delta , \bar t + \delta] \times [0, \epsilon]$ by $\Psi ^{-1}$. Without lost of generality, we can assume $p _n \in W (q)$ and $p _n = \beta _{t_n} (s_n)$ for some $t_n \in  [\bar t - \delta , \bar t + \delta]$ and $s_n \in (0, \epsilon)$ such that $t_n \to \bar t$ and $s_n \to 0$ as $n \to \infty$.

Fix $n$ large enough and, for notational convenience, reparametrise $C$ such that $t_n =0 $ and $s_n = s(0) \in (0, \epsilon)$. Let $\tilde \gamma $ be the component of $\Psi ^{-1} (\gamma)$ containing $ \Psi ^{-1} (p_n) = (0, s(0))$. $\tilde \gamma$ is a geodesic  that intersects the vertical geodesic $\tilde \beta _{0} := \set{0} \times [0,\epsilon]$ at the point $( 0 ,s(0))$. Observe that $\tilde \gamma $ cannot be locally above the equidistant $\tilde C _{s(0)}$ at $(0 ,s(0))$ since the geodesic curvature of the equidistant $\tilde C _{s(0)} $ is positive, the curvature vector is pointing up.

Thus, we can choose one of the components of $\tilde \gamma  \setminus \set{(0 , s(0))}$ so that the tangent vector to the component at $(0, s(0))$ is going down. Orient this component, $\tilde \gamma  ^+(t) $, in the direction it is going down at $(0 ,s(0))$. Without loss of generality, we can assume $\tilde \gamma  ^+(t)$ goes down for $t \in (0 , \omega)$, $\omega >0$ small.

Then, $\tilde \gamma ^+(t)$ is a geodesic arc making an (exterior) angle $\theta  \in [\pi /2 , \pi )$ with the vertical geodesic $\tilde \beta _{0} := \set{0} \times [0,\epsilon]$, observe that $\theta < \pi$ since $\gamma$ is embedded. Using that the geodesic curvature of the equidistants $\tilde C _{s}$ is positive, the curvature vector is pointing up, it is easy to observe, if we write $\tilde \gamma  ^+(t) = (t, s(t))$, that $s(t)$ is monotonically decreasing in $[0 , \delta]$. In fact, the above argument remains true for all $0< t$ and hence, we consider the half-geodesic $\tilde \gamma  ^+ : [0 , + \infty) \to \r \times [0 , \epsilon] $. 

$\tilde \gamma ^+$ strictly descends for all time (since it cannot become tangent from above to an equidistant), so it is defined for all time and is asymptotic to some $\r \times \set{\bar s}$, $\bar s \in [ 0, \epsilon )$, as $t$ diverges. Thus, the exterior angle $\alpha (t)$ that $\tilde \gamma $ makes with the vertical can be made as close to $\pi /2 $ as desired, for $t$ large. 

Fix $L \geq \delta  $ and let $D (L)$ be the disk bounded by $\left(\set{0} \times [0,s(0)]\right) \cup \tilde \gamma  ^+ \left( [0, L]\right) \cup \left( \set{ L} \times [0,s (L)] \right) \cup \left([0, L] \times \set{0} \right)$, then by the Gauss-Bonnet Theorem, we get 
$$ \int _{D(L)} \tilde K + \theta   + \alpha  (L)= \pi ,$$where $\theta  $ and $\alpha  (L)$ are the exterior angles at $(0 ,s(0))$ and $( L , s (L))$, respectively, and $\tilde K$ the Gauss curvature in $\left( \r \times [-\epsilon , \epsilon] , \tilde g\right)$  (cf. Figure \ref{Geod2}). Therefore, using that $\tilde K \geq A > 0$, we obtain
$$ 0 <  \int _{D (\delta)} \tilde K  \leq \pi - \theta  - \alpha  (L)  \leq \pi /2 - \alpha  (L) \to 0,$$as $L \to + \infty$; which is a contradiction.

\begin{figure}[htbp]
\begin{center}
\includegraphics[scale=0.4]{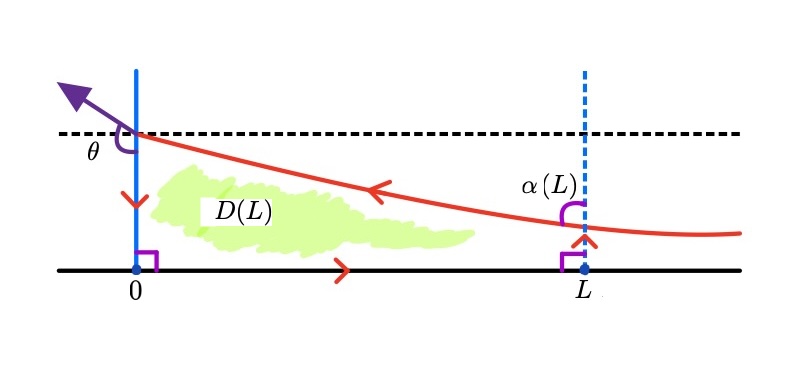}
\caption{Disk bounded by $\left(\set{0} \times [0,s(0)]\right) \cup \tilde \gamma  ^+ \left( [0, L]\right) \cup \left( \set{ L} \times [0,s (L)] \right) \cup \left([0, L] \times \set{0} \right)$}
\label{Geod2}
\end{center}
\end{figure} 
\end{proof}

We can extend this theorem to complete, non-compact, orientable surfaces $\amb$ of positive curvature $\Kb >0$. Suppose $\gamma$ is a complete, non-compact, embedded geodesic of $\amb$. It is not hard to show that there exists an $\epsilon >0$ such that the $\epsilon-$tube ${\rm Tub}_\epsilon (\gamma) $ is embedded in $\amb$. Our proof of Theorem A shows that if a geodesic $\tilde \gamma$ of $\amb$ intersects ${\rm Tub}_\epsilon (\gamma) $, then $\tilde \gamma$ intersects $\gamma$. The proof also shows that a geodesic ray on a strictly convex sphere is not embedded.

In \cite{RLanHRos96} the authors claim that Theorem A should extend to dimensions $n>2$: {\it When $\amb ^n$ is closed and with positive sectional curvatures, a complete embedded geodesic $\gamma \subset \amb$ must intersect every closed totally geodesic hypersurface of $\amb$.} However, the proof is not complete. We do not know if this holds for $n=3$.  

%Since $b : [t_n , +\infty) \to [0, + \infty)$ is monotonically decreasing, then $\lim _{t \to + \infty} b(t) = \bar b \geq 0$ and $\lim _{t \to + \infty} b' (t) =0 $; which implies that $\lim _{L \to +\infty} \alpha (L) = \pi /2$. 

%Fix $\tau >0$ as small as we want, since $s_n \to 0$ (up to a subsequence) there exists $n_0 \in \mathbb{N}$ such that $\theta _n \in [\pi /2 , \pi /2 + \tau )$ for all $n \geq n_0$. Fix $L \geq \delta$,  then $\tilde \gamma _n ^+ \left([t_n , + \infty ) \right)$ intersects the vertical geodesic $\tilde \beta _{t_n+ L} := \set{t_n + L} \times [0,\epsilon]$ at a point $(t_n + L ,  s _n (L))$. Let $D_n (L)$ be the disk bounded by $\left(\set{t_{n}} \times [0,s_n]\right) \cup \tilde \gamma _n ^+ \left( [t_n , L]\right) \cup \left( \set{t_{n} + \delta } \times [0,\tilde s_n] \right) \cup \left([t_n , t_n + \delta] \times \set{0} \right)$, then by Gauss-Bonnet Theorem, we get 

\section*{Acknowledgments}

The authors would like to thank Brian White and Joaqu\'{i}n P\'{e}rez for their numerous thoughtful suggestions and improvements to the paper.

The first author, Jos\'e M. Espinar, is partially supported by the Maria de Maeztu Excellence Unit IMAG, reference CEX2020-001105-M, funded by MCINN/AEI/10.13039/501100011033/CEX2020-001105-M; Spanish MIC Grant PID2020-118137GB-I00 and MIC-NextGenerationEU Grant 30.RP.23.00.04 CONSOLIDACION2022.

\end{document}